\documentclass[12pt]{article}   
 \usepackage{amsmath,amssymb,latexsym,footnote,dsfont}
 \usepackage{graphicx,pstricks}%,showkeys}
\usepackage[left=2cm,right=2cm,top=2cm,bottom=2cm]{geometry}
\usepackage{mathrsfs}
\usepackage{amsfonts}
\usepackage{amsthm}
\usepackage{amsfonts}
\usepackage[colorlinks=true, pdfstartview=FitV, linkcolor=blue, citecolor=blue, urlcolor=blue]{hyperref}
\makeatletter  
\def\@@insvline#1#2{{\setbox0\hbox{\m@th$#1\mathrm I$}  
  \rlap{\m@th$#1 \mkern 5mu  
  \vrule height.95\ht0 depth-.005\ht0 width.09\ht0 $}  
  {\mathrm #2} }}

\def\Q{\mathpalette\@@insvline{Q}}

\makeatother  
  \newtheorem{defi}{Definition}
  \newcommand{\bd}{\begin{defi}} 
  \newcommand{\ed}{\end{defi}}

  \newtheorem{lemm}[defi]{Lemma}  
  \newcommand{\bl}{\begin{lemm}}
  \newcommand{\el}{\end{lemm}} 

  \newtheorem{theo}[defi]{Theorem}
  \newcommand{\bt}{\begin{theo}}
  \newcommand{\et}{\end{theo}}

  \newtheorem{cor}[defi]{Corollary}
  \newcommand{\bc}{\begin{cor}}
  \newcommand{\ec}{\end{cor}}

  \newtheorem{pro}[defi]{Proposition}
  \newcommand{\bp}{\begin{pro}}
  \newcommand{\ep}{\end{pro}}
  \makeatletter
  \def\proof{\@ifnextchar[\opargproof{\opargproof[\bf Proof \hfil\\ ]}}
  \def\opargproof[#1]{\par\noindent {\bf #1 }}
  
  \makeatother
   \newcommand{\bydef}{\overset{def}{=}}
\DeclareFontFamily{OT1}{nice}{}
\DeclareFontShape{OT1}{nice}{m}{n}{<5> <6> <7> <8> <9> <10>
<12><10.95><14.4><17.28><20.74><24.88>callig15}{}
\DeclareFontFamily{U}{nice}{}
\DeclareFontShape{U}{nice}{m}{n}{<5> <6> <7> <8> <9> <10>
<12><10.95><14.4><17.28><20.74><24.88>callig15}{}
\DeclareSymbolFont{calligra}{U}{nice}{m}{n}
\DeclareSymbolFontAlphabet{\nice}{calligra}
\DeclareFontFamily{OT1}{cmdh}{}
\DeclareFontShape{OT1}{cmdh}{m}{n}{<10>cmdunh10}{}
\input umsa.fd
\input umsb.fd
\input ueuex.fd
\input ueuf.fd
\input ueur.fd
\input ueus.fd

\def\epsilon{\varepsilon}  
\def\phi{\varphi}

\def\div{\operatorname{div}}
\def\curl{\operatorname{curl}}

\newtheorem{rmq}{Remark}
\newtheorem{lemma}{Lemma}
\newtheorem{prop}{Proposition}
\newtheorem{defin}{Definition}
\newcommand{\norm}[1]{\left\Vert #1\right\Vert}

\begin{document} 

\title{Uniqueness result for the 3-D Navier-Stokes-Boussinesq Equations with Horizontal Dissipation}
\bigskip

%\date{\today}

\author{ Pierre Dreyfuss\thanks{Universit\'e C\^ote d’Azur, CNRS, LJAD, France (email: pierre.dreyfuss@univ-cotedazur.fr).}, Haroune Houamed\thanks{Universit\'e C\^ote d’Azur, CNRS, LJAD, France (email: haroune.houamed@univ-cotedazur.fr).} \footnote{Corresponding author.}} 

\maketitle

\abstract{\noindent In this paper, for the 3-D Navier-Stokes-Boussinesq system with horizontal dissipation, where there is no smoothing effect on the vertical derivatives, we prove a uniqueness result of solutions $ (u,\rho)\in L^{\infty}_T\big( H^{0,s}\times H^{0,1-s}\big)$ with  $ (\nabla_h u,\nabla_h\rho)\in L^{2}_T\big( H^{0,s}\times H^{0,1-s}\big)$ and $s\in [\frac{1}{2},1]$. As a consequence, we improve the conditions stated in the paper \cite{Miao} in order to obtain a global well-posedness result in the case of axisymmetric initial data.
} 

\vspace{0.2 cm} 

\noindent 
{\it keywords:}
Boussinesq system, Horizontal dissipation, Anisotropic inequalities, Uniqueness, 
Global well-posedness. 
%\footnotetext{1991 Mathematics Subject 
%Classification: 35J70,35B65}
%\bigskip\bigskip
 \tableofcontents
 
\section{Introduction and main results}

The Navier-Stokes-Boussinesq system is obtained from the density dependent Navier-Stokes equations by using the Boussinesq approximation. It is widely used to model geophysical flows 
(for instance oceanical or atmospherical flows) whenever rotation and stratification play an important role (see \cite{pedlo}). We will consider the following so-called Navier-Stokes-Boussinesq equations with horizontal dissipation:\\
\begin{equation}\label{NSB_h}
  \left\{\begin{array}{l}
\big(\partial_t+u\cdot\nabla  \big)u -\Delta_hu+\nabla P=\rho e_3, \quad \text{in } \mathbb R^+\times \mathbb R^3,\\
\big(\partial_t+u\cdot\nabla  \big)\rho -\Delta_h\rho =0,\\
\div u=0,\\
(u,\rho)_{|t=0}=(u_0,\rho_0),
\end{array}\right. \tag{$NSB_h$} 
\end{equation}
where $\Delta_h \bydef \partial_1^2 + \partial_2^2$ denotes the horizontal laplacian and 
$e_3=(0,0,1)^T$ is the third vector of the canonical basis of $\mathbb R^3$. \\ 
The unknowns of the system are $u=(u^1,u^2,u^3)$, $\rho$ and $P$ which represent respectively: the velocity, the density and the pressure of the fluid. \\ 

\noindent In the following we will say that $(u,\rho)$ is a solution to $\eqref{NSB_h}$ if it is a weak solution in the classical sense (see for instance \cite{Chemin} pages 123,132 and 204). We recall also that from a solution $(u,\rho)$ we may use a result of De Rham in order to recover a pressure $P$ (which depends on $u$ and $\rho$) and to obtain a distributional solution 
$(u,\rho,P)$ of the system \eqref{NSB_h}. \\ 

\noindent Note that in \eqref{NSB_h} the diffusion only occurs in the horizontal direction. This is a natural assumption for several cases of interest in geophysical fluids flows (see \cite{pedlo}). However $-\Delta_h$ is a less regularizing operator than the laplacian $-\Delta$ and we cannot expect a better theory than for the classical Navier-Stokes-Boussinesq equations: 
\begin{equation}\label{NSB}
 \left\{\begin{array}{l}
\big(\partial_t+u\cdot\nabla  \big)u -\Delta u+\nabla P=\rho e_3, \quad \text{in } \mathbb R^+\times \mathbb R^3,\\
\big(\partial_t+u\cdot\nabla  \big)\rho -\Delta\rho =0,\\
\div u=0,\\
(u,\rho)_{|t=0}=(u_0,\rho_0).
\end{array}\right.\tag{$NSB$}
\end{equation}
\noindent In particular, the question of the global well-posedness of \eqref{NSB} and consequently of \eqref{NSB_h} remains 
largely open, but recently the system \eqref{NSB_h} has received a lot of attention from mathematicians (see for instance \cite{Miao,Adhikari,Wu}) and significant progress in its analysis have been made. See also \cite{HHZ, Hmidi1, Hmidi2, HZ} for more related results for other models of the Boussinesq system.\\ 

\noindent Note again that \eqref{NSB_h} involves the operator $-\Delta_h$ which smooth only along the horizontal variables. Hence, we need to estimate differently the horizontal and the vertical directions,  and the natural functional setting for the analysis involves some anisotropic Sobolev and Besov spaces. The definitions of these spaces and some of their important properties are recalled in the next section. \\ 

\noindent In order to analyse \eqref{NSB_h} it is useful to forget its second equation for a while, and 
to consider first the Navier-Stokes equations with horizontal laplacian:
\begin{equation}
\quad \left\{\begin{array}{l}
\big(\partial_t+u\cdot\nabla  \big)u -\Delta_h u+\nabla P=0, \quad \text{in } \mathbb R^+\times \mathbb R^3,\\
\div u=0,\\
u_{|t=0}=u_0.
\end{array}\right.\tag{$NS_h$}
\end{equation}
\noindent Several interesting studies for this last system were done. In \cite{Chemin}, the authors proved the local existence and the global one for small data in $H^{0,s}$ for some $s>\frac{1}{2}$. The proof of the existence part in \cite{Chemin} uses deeply the structure of the equation and the fact that $u$ is a divergence free vector field. The key point used in their estimates is related to the fact\footnote{Recall that this argument permits to prove the uniqueness of weak-solution for the classical Navier-Stokes problem in dimension two.} that $H^{\frac{1}{2}}(\mathbb{R}^2) \hookrightarrow L^4(\mathbb{R}^2)$ and that $H^s(\mathbb{R})$ is an algebra. Hence, it is easy to deal with the term $u^h\cdot \nabla_h u$ by using some product rules in the well-chosen spaces. Next, after using the divergence free condition together with some Littlewood-Paley stuffs in a clever way they were able to treat the term $u^3\partial_3 u$ with the same argument. Always in \cite{Chemin}, the authors proved also a uniqueness result (but only for $s>\frac{3}{2}$, because of the term $w^3\partial_3 u$) by establishing a $H^{0,s_0}$-energy estimate for a difference between two solutions $w=u-v$, where $s_0\in]\frac{1}{2},s]$. Later, in \cite{Ifti3}, D.Iftimie had overcome the difficulty by remarking that it is sufficient to estimate $w$ in $H^{-\frac{1}{2}}$ with respect to the vertical variable, and this only requires an $H^{\frac{1}{2}}$ regularity for $u$ in the vertical direction. Then he proved a uniqueness result for any $s>\frac{1}{2}$, and the gap between existence and uniqueness was closed. \\   

\noindent To do something similar with system \eqref{NSB_h}, we begin by estimating the horizontal terms (terms which contain only horizontal derivatives) by using some product rules in the adequate Besov and Sobolev spaces. For the vertical terms (terms which contain only vertical derivatives) we follow in general the idea in \cite{Chemin} in order to transform them into terms similar to the horizontal ones by using the divergence free condition. Hence, for $s\in]\frac{1}{2},1]$, we first propose to estimate the difference between two solutions $w=u-v$ in $H^{0,s-1}$ instead of $H^{0,-\frac{1}{2}}$ providing that the solution $u$ already exists in the $H^{0,s}$ energy-space (see Appendix for a proof of an existence result). For the second equation, denoting the difference between two solutions $\theta=\rho_1-\rho_2$ we remark that:
\begin{itemize}
\item The function $\rho$ only appears in the third equation of $u$ (the equation for the component $u^3$). Hence, a priori, we only need to estimate $\rho$ in the $H^{s-1}$-norm with respect to the vertical variable.  
\item In order to deal with the term $u^h \cdot\nabla_h \theta$, we must estimate $\theta$ with respect to the vertical variable in some space $H^{-\alpha}$, with $\alpha\geq 0$ and such that $H^s(\mathbb{R}) \times H^{-\alpha}(\mathbb{R})$ holds to be a subspace of $H^{-\alpha}(\mathbb{R})$. In fact, Lemma \ref{productrule1} bellow says that the minimum index $-\alpha$ that can be chosen is $-\alpha=-s$. 
\item  For the term $w^h \cdot \nabla_h \rho$, if we consider that $\rho$ lies in some $H^{\beta}$-space, with respect to the vertical variable, then a direct application of the product rules shows that we need  
$\beta\geq 1-s$. Moreover, because the system is hyperbolic in the vertical direction, we expect the loss 
of one derivative. 
\end{itemize}
Hence, we will estimate vertically $\rho$ in $H^{1-s}$ and $\theta$ in $H^{-s}$. \\ 

\noindent For the critical case where $s=\frac{1}{2}$, in \cite{Paicu2} M.Paicu proved a uniqueness\footnote{We should mention that the existence of solution in such scaling-invariant space is still an open problem even for the classical Navier-Stokes system.} result for $(NS_h)$ in $L^{\infty}_T(H^{0,\frac{1}{2}}) \cap L^2_T (H^{1,\frac{1}{2}})$. It is clear that such a space falls to be embedded in $L^{\infty}$ in the vertical direction which is the major problem that prevents using similar arguments to those in the case where $s>\frac{1}{2}$. In order to prove the uniqueness, the author in \cite{Paicu2} established a double logarithm estimate (see \eqref{cases=1/2}) and concluded by using the Osgood's Lemma. We will provide some details to adapte the idea of \cite{Paicu2} and apply it to \eqref{NSB_h}. \\ 
  
\noindent Our main result is the following:
\begin{theo}{\textbf{(Uniqueness)}} \label{th1} \\
Let $s\in[\frac{1}{2},1]$ and $(u,\rho),(v,\eta)$ be two solutions for system \eqref{NSB_h} in $$L^{\infty}_{loc}(\mathbb{R}_+;H^{0,s})\cap L^{2}_{loc}(\mathbb{R}_+;H^{1,s})\times L^{\infty}_{loc}(\mathbb{R}_+;H^{0,1-s})\cap L^{2}_{loc}(\mathbb{R}_+;H^{1,1-s}). $$
Then $(u,\rho)=(v,\eta).$
\end{theo}   

\noindent As an interesting consequence, we can improve the results of global well-posedness in the case of axisymmetric initial data established in \cite{Miao}. \\ 

\noindent Let us first recall some basic notions: We say that a vector field $u$ is axisymmetric if it satisfies 
$$\mathcal{R}_{-\alpha}(u(\mathcal{R}_{\alpha}(x))) = u(x), \quad \forall \alpha\in[0,2\pi], \ \forall x \in \mathbb R^3,$$
where $\mathcal{R}_{\alpha}$ denotes the rotation of axis $(Oz)$ and with angle $\alpha$. Moreover, an axisymmetric 
vector field $u$ is called without swirl if it has the form: 
$$u(x)=u^r(r,z)e_r + u^z(r,z)e_z, \quad x=(x_1,x_2,x_3),\quad r=\sqrt{x_1^2+x_2^2} \text{ and } z=x_3.$$
We say that a scalar function $f$ is axisymmetric, if the vector field $x\mapsto f(x)e_z$ is axisymmetric. We also denote by $\omega=\curl{u}$ the vorticity of $u$. Then we will prove: 

\begin{theo}{\textbf{(Global well-posedness)}} \label{th2}\\
Let $u_0\in H^1(\mathbb{R}^3)$ be an axisymmetric divergence free vector field without swirl such that $\frac{\omega_0}{r}\in L^2$ and let $\rho_0\in L^2$ be an axisymmetric function. Then there exists a unique global solution $(u,\rho)$ of the system \eqref{NSB_h}. Moreover, we have:  
\begin{align*}
u &\in \mathcal{C}(\mathbb{R}_+;H^1)\cap L^2_{loc}(\mathbb{R}_+;H^{1,1}\cap H^{2,0}), \quad 
\frac{\omega}{r} \in L^{\infty}_{loc}(\mathbb{R}_+;L^2)\cap L^2_{loc}(\mathbb{R}_+;H^{1,0}), \\
\rho &\in \mathcal{C}(\mathbb{R}_+;L^2)\cap L^2_{loc}(\mathbb{R}_+;H^{1,0}).
\end{align*}
\end{theo} 

\noindent Note that in Theorem \ref{th2}, we only assume that $(u_0,\rho_0)\in H^1\times L^2$ whereas in \cite{Miao} the authors consider a stronger condition. Namely, in addition of the hypothesis of Theorem \ref{th2} they assume that $(\nabla \times u_0,\rho_0)$ is in $H^{0,1}\times H^{0,1}$ or in $L^{\infty} \times H^{0,1}$. 
In both works the key point consists to establish an uniqueness result: it is the Theorem \ref{th1} for us, whereas in \cite{Miao} the authors assume a strong initial condition in order to obtain some double exponential control in time for the gradient of $u$. \\

\noindent The paper is organized as follows: in section 2, for the reader's convenience, we recall the required background concerning the functional spaces and some useful technical tools. In section 3, we establish several a priori estimates which are then used in section 4 to prove the two theorems above. Finally, in Appendix we shall prove a result of well posedness for \eqref{NSB_h} under some smallness conditions involving only $T$, the $L^2$-norm of $\rho_0$ and the $H^{0,s}$-norm of $u_0$, the uniqueness part of this result is a consequence of Theorem 1.

\section{Functional framework}
\subsection{Notations and functional spaces}

\noindent Throughout this paper we write $\mathbb R^3=\mathbb R^2_h\times\mathbb R_v$ and for any vector $\xi=(\xi_1,\xi_2,\xi_3)\in \mathbb R^3$, we will denote the two first components by $\xi_h$ and the last one by $\xi_v$, that is to say: $\xi=(\xi_1,\xi_2,\xi_3)\bydef (\xi_h,\xi_v)$. Similarly, for any vector field $X=(X^1,X^2,X^3)$ we will write $X=(X^h,X^v)$ with the meaning that $X^h=(X^1,X^2)$ and $X^v=X^3$. \\ 

\noindent We will also use the notations: 
$$
H_h^s=H^s(\mathbb R^2_h),\ H_v^s=H^s(\mathbb R_v),\ L_v^p(H_h^s) = L^p(\mathbb R_v;H_h^s) \text{ and }
L_T^rL_h^pL_v^q = L^r(0,T;L^p(\mathbb R^2_h;L^q(\mathbb R_v))).
$$   
 
%\noindent In order to introduce some anisotropic Sobolev and Besov spaces we need to recall the anisotropic dyadic decomposition (see \cite{Chemin1,Chemin2,Ifti2,Chemin} for more details). \\

\noindent Recall that \eqref{NSB_h} involves the operator $-\Delta_h$ which only regularizes along the horizontal direction. Hence, the regularity along the vertical variable must be measured differently than the horizontal ones, which motivates the consideration of anisotropic functional spaces. We now provide the definition of these spaces whose are based on an anisotropic version of the Littlewood-Paley theory (see \cite{Chemin1, Haroune, Ifti1, Ifti2} for more details and more examples of applications). \\  

\noindent Let $(\psi,\varphi)$ be a couple of smooth functions with value in $[0,1]$ satisfying: 
%$\text{Supp} \psi \subset \{ \xi \in \mathbb{R} : |\xi| \leq \frac{4}{3}\}$ and in the shell $\{\xi \in \mathbb{R} :\frac{3}{4} \leq |\xi| \leq \frac{8}{3}  \}$ and:
\begin{alignat*}{2}
&\text{Supp } \psi \subset \{ \xi \in \mathbb{R} : |\xi| \leq \frac{4}{3}\}, 
\quad &&\text{Supp } \varphi \subset \{\xi \in \mathbb{R} :\frac{3}{4} \leq |\xi| \leq \frac{8}{3} \}, \\
&\psi(\xi) + \sum_{q\in \mathbb{N}} \varphi(2^{-q}\xi) = 1 \;\; \forall \xi \in \mathbb{R}, 
\quad 
&&\sum_{q\in \mathbb{Z}} \varphi(2^{-q}\xi) = 1 \;\; \forall \xi \in \mathbb{R}\backslash \{0\}.
\end{alignat*}
Let $a$ be a tempered distribution, $\hat{a}=\mathcal{F}(a)$ its Fourier transform and $\mathcal{F}^{-1}$ 
denotes the inverse of $\mathcal{F}$. We define the non-homogeneous dyadic blocks $\Delta_q$ and the homogeneous ones $\dot{\Delta}_q$ by setting:
\begin{alignat*}{2}
&\Delta^v_{q} a \bydef \ \left\{ \begin{array}{l} \mathcal{F}^{-1}\big(\varphi (2^{-q}|\xi_3| \hat{a} ) \big)\;\text{for} \; q\in \mathbb{N}, \\ 
\mathcal{F}^{-1}\big(\psi (|\xi_3| \hat{a} ) \big) \; \text{for} \; q=-1, \\ 
0 \; \text{for} \; q\leq-2, \end{array} \right. \quad 
&&\Delta^h_j a \bydef \ \left\{ \begin{array}{l} \mathcal{F}^{-1}\big(\varphi (2^{-j}|\xi_h| \hat{a} ) \big)\; \text{for} \; j\in \mathbb{N}, \\ 
\mathcal{F}^{-1}\big(\psi (|\xi_h| \hat{a} ) \big) \; \text{for} \; j=-1, \\ 
0 \; \text{for} \; j\leq-2, \end{array} \right. \\
&S_q^v \bydef \sum_{m<q}\Delta_{m}^v,  \quad \forall q \in \mathbb Z, \quad 
&&S_j^h \bydef \sum_{m<j}\Delta_{m}^h,  \quad \forall j \in \mathbb Z, \\ 
&\dot{\Delta}^v_q a \bydef \mathcal{F}^{-1}\big(\varphi (2^{-q}|\xi_3| \hat{a} ) \big),\; \forall \; q\in \mathbb{Z}, 
\quad &&\dot{\Delta}^h_j a \bydef \mathcal{F}^{-1}\big(\varphi (2^{-j}|\xi_h| \hat{a} ) \big),\; \forall \; j\in \mathbb{Z}, \\
&\dot{S}_q^v \bydef \sum_{m<q}\dot{\Delta}_{m}^v ,  \quad \forall q \in \mathbb Z, \quad 
&&\dot{S}_j^h \bydef \sum_{m<q}\dot{\Delta}_{j}^h ,  \quad \forall j \in \mathbb Z.
\end{alignat*}
We then have $a = \sum_{m\geq -1} \Delta_m a =\sum_{m\in \mathbb{Z}} \dot{\Delta}_m a$ for both horizontal and vertical decomposition. Moreover, in all the situations, i.e. for 
$\Delta,S$ with the same index of direction (horizontal or vertical) and in both homogeneous and non-homogeneous cases they hold:  
\begin{align*}
&\Delta_m\Delta_{m'} a =0, \; \text{ if} \; |m-m'| \geq 2, \\
&\Delta_m\big(S_{m'-1}a\Delta_{m'} a\big) =0, \; \text{ if} \; |m-m'| \geq 5, \\
&\Delta_m\sum_{i\in \{0,1,-1 \}}\sum_{m'\in \mathbb{Z}}(\Delta_{m'+i}a\Delta_{m'} a\big)=\Delta_m\sum_{i\in \{0,1,-1 \}}\sum_{m'\geq m-N_0}(\Delta_{m'+i}a\Delta_{m'} a\big),
\end{align*}
where $N_0\in \mathbb N$ can be chosen independently of $a$ (we can take $N_0=5$). \\ 
 
\noindent In what follows, we will use the so-called Bony's decomposition (see \cite{Chemin1}): 
\begin{alignat*}{2}
&ab= T_a(b) + T_b(a) + R(a,b), \quad &&\\
&T_a(b)\bydef \sum_{q\in Z} S_{q-1} a \Delta_q b, \quad 
&&R(a,b)\bydef \sum_{i\in \{0,1,-1 \}}\sum_{q\in Z} \Delta_{q+i} a \Delta_q b.
\end{alignat*}
Here again all the situations may be considered however particular cases must be precised by using the adequate 
notations. For instance, if we consider the non-homogeneous version for the vertical variable, we have to 
add the exponent $^{v}$ in all the operators $T_a,T_b,R,S_q$ and $\Delta_q$. \\     

\noindent Our analysis will be made in the context of the non-homogeneous and anisotropic Sobolev and Besov spaces:
\begin{defin}
Let $s,t$ be two real numbers and let $p,q_1,q_2$ be in $[1,+\infty]$, we define the space $(B^t_{p,q_1})_h(B^s_{p,q_2})_v$ as the space of tempered distributions $u$ such that 
$$
\norm u _{(B^t_{p,q_1})_h(B^s_{p,q_2})_v}\bydef \norm { 2^{kt}2^{js} \norm {\Delta_k^h \Delta_j^v u}_{L^p}}_{\mathscr{l} \ell _k^{q_1}(\mathbb{Z};\ell_j^{q_2}(\mathbb{Z})) } < \infty .
$$
In the situation where $q_1=q_2=q$, we use the notation $B_{p,q}^{t,s}\bydef(B^t_{p,q})_h(B^s_{p,q})_v$. If $p=q=2$ then this last space is denoted by $H^{t,s}$. If moreover, $t=0$ then we have:
$$
\norm u _{H^{0,s}} \approx \big(\sum_{j\in \mathbb{Z}} 2^{2js} \norm { \Delta_j^v u}_{L^2}^{2}\big)^{\frac{1}{2}}.
$$
\end{defin}
\noindent Let $f\in B_{p,2}^{0,s}$, the following properties will be of constant use in the paper: 
\begin{align}
 &\forall s\in\mathbb R, \ \exists c_q=c_q(f): \; \left\|\Delta_q^v f  \right\|_{L^p}\leq c_q 2^{-sq} \left\|f \right\|_{B_{p,2}^{0,s}},\ \text{and} \,\displaystyle\sum_{q\geq -1}c_q^2 \leq 1, \label{stat} \\
&\forall s<0, \ \exists \tilde{c}_q=\tilde{c}_q(f): \; \left\|S_q^v f  \right\|_{L^p}\leq \tilde{c}_q 2^{-sq} \left\|f \right\|_{B_{p,2}^{0,s}},\ \text{and} \,\displaystyle\sum_{q\geq -1}\tilde{c}_q^2 \leq 1. \label{stat2}
\end{align}
\subsection{Some useful Lemmata}
Other properties of the spaces defined in the previous subsection can be found in \cite{Chemin1} for the usual isotropic version, and in \cite{Chemin2} for the anisotropic case. A very helpful tool related to the Bernstein Lemma (see for instance Lemma 2.1 in \cite{Chemin2}) is given by the following Lemma.
\begin{lemma} \label{ber}
Let $\mathcal{B}_h$ (resp. $\mathcal{B}_v$) be a ball of $\mathbb{R}^2_h$ (resp. $\mathbb{R}_v$) and $\mathcal{C}_h$ (resp. $\mathcal{C}_v$) a ring of $\mathbb{R}^2_h$ (resp. $\mathbb{R}_v$). Let also $a$ be a tempered distribution and $\hat{a}$ its Fourier transform. Then for $1\leq p_2\leq p_1 \leq \infty$ and $1\leq q_2\leq q_1\leq \infty$ we have:
\begin{align*}
&\text{Supp }\hat{a} \subset 2^k\mathcal{B}_h \ \Longrightarrow, \ 
\norm {\partial^{\alpha}_{x_h}a}_{L^{p_1}_h(L^{q_1}_v)} \lesssim 2^{k\big( |\alpha| + 2\big( \frac{1}{p_2}- \frac{1}{p_1}\big) \big)} \norm {a}_{L^{p_2}_h(L^{q_1}_v)}, \\ 
&\text{Supp }\hat{a} \subset 2^l\mathcal{B}_v \ \Longrightarrow \  
\norm {\partial^{\beta}_{x_3}a}_{L^{p_1}_h(L^{q_1}_v)} \lesssim 2^{l\big( \beta + \big( \frac{1}{q_2}- \frac{1}{q_1}\big) \big)} \norm {a}_{L^{p_1}_h(L^{q_2}_v)}, \\
&\text{Supp }\hat{a} \subset 2^k\mathcal{C}_h \ \Longrightarrow \
\norm {a}_{L^{p_1}_h(L^{q_1}_v)} \lesssim 2^{-kN} \sup_{|\alpha|=N}\norm {\partial^{\alpha}_{x_h}a}_{L^{p_1}_h(L^{q_1}_v)}, \\
&\text{Supp }\hat{a} \subset 2^l\mathcal{C}_v \ \Longrightarrow \
\norm {a}_{L^{p_1}_h(L^{q_1}_v)} \lesssim 2^{-lN} \norm {\partial^{N}_{x_3}a}_{L^{p_1}_h(L^{q_1}_v)}.
\end{align*}
\end{lemma}
\noindent We will also need some product rules in (non-homogeneous) Sobolev spaces, which we prove it here.
\begin{lemma}\label{productrule1}
Let $\sigma,\sigma',s,s_0 \in \mathbb{R}$ verifying $\sigma,\sigma'<1,\ \sigma+\sigma'> 0,\ s_0>\frac{1}{2}, s\leq s_0$ and $s+s_0\geq0$ then 
there exists a constant $C=C(\sigma,\sigma',s,s_0)$ such that:
$$
\norm {ab}_{H^{\sigma+\sigma'-1,s}}\leq C \norm {a}_{H^{\sigma,s}} \norm {b}_{H^{\sigma',s_0}}, \quad \forall a,b \in \mathcal{S}.
$$
\end{lemma}
\begin{proof}
Remark first that because $\norm{ab}_{H^{\sigma+\sigma'-1,s}}=\norm{\norm{ab}_{H^s_v}}_{H^{\sigma+\sigma'-1}_h}$,
we have only to prove that:
\begin{equation}\label{productdim1}
H^s(\mathbb{R}) \cdot H^{s_0}(\mathbb{R}) \subset H^{s}(\mathbb{R}).
\end{equation}
Indeed, by using (\ref{productdim1}) together with the usual product rules with respect to the horizontal variables, the 
desired result follows (see for instance \cite{Chemin1}). Note also that when $s_0>s>\frac{1}{2}$, the inclusion (\ref{productdim1}) is trivial since in this case the space $H^{s}$ is an algebra and clearly $H^{s_0}\hookrightarrow H^{s}$.\\
It remains then only to prove (\ref{productdim1}) in the situation $s_0>\frac{1}{2}\geq s$ and $s+s_0\geq 0$. In order to do this, we use the Bony's decomposition in the vertical variable: $ab= T_a^vb + T_b^va + R^v(a,b)$. \\ 
For the first term, let us consider the two cases: $s<\frac{1}{2}$ and $s=\frac{1}{2}$.\\
\textit{\textbf{The case}} $s<\frac{1}{2}$: By using the embedding $H^s(\mathbb{R})\hookrightarrow B^{s-\frac{1}{2}}_{\infty,2}(\mathbb{R})$, together with \eqref{stat2} we obtain for any $q\geq -1$: 
\begin{align*}
\norm {\Delta^v_q( T_a^vb)}_{L^2(\mathbb{R})}&\lesssim \norm {S^v_{q-1}a}_{L^{\infty}(\mathbb{R})} \norm {\Delta_q^v b}_{(L^2\mathbb{R})} \lesssim c_q^2 2^{-q(s-\frac{1}{2})} 2 ^{-qs_0} \norm a_{B^{s-\frac{1}{2}}_{\infty,2}(\mathbb{R})}\norm b_{H^{s_0}(\mathbb{R})} \\
&\lesssim c_q^2 \max\{ 1, 2^{s_0-\frac{1}{2}} \} 2 ^{-qs} \norm a_{H^s(\mathbb{R})}\norm b_{H^{s_0}(\mathbb{R})} \lesssim c_q^2  2 ^{-qs} \norm a_{H^s(\mathbb{R})}\norm b_{H^{s_0}(\mathbb{R})}.
\end{align*}
It follows that 
\begin{equation}\label{T_ab}
\norm{ T_a^vb}_{H^s(\mathbb{R})} \lesssim \norm a_{H^s(\mathbb{R})}\norm b_{H^{s_0}(\mathbb{R})}.
\end{equation} 
\textit{\textbf{The case}} $s=\frac{1}{2}$. We use the following estimate:
$$
\norm {S^v_{q-1}a}_{L^{\infty}(\mathbb{R})}\leq \sum_{-1\leq j \leq q} 2^{\frac{j}{2}} \norm {\Delta_j^va}_{L^{2}(\mathbb{R})}\lesssim \sqrt{q} \norm a_{H^{\frac{1}{2}}(\mathbb{R})},
$$
in order to obtain:
\begin{equation*}
\norm {\Delta^v_q( T_a^vb)}_{L^2(\mathbb{R})}\lesssim \norm {S^v_{q-1}a}_{L^{\infty}(\mathbb{R})} \norm {\Delta_q^v b}_{L^2(\mathbb{R})} \lesssim c_q \sqrt{q} 2^{-q(s_0-\frac{1}{2})} 2 ^{-\frac{q}{2}} \norm a_{H^{\frac{1}{2}}(\mathbb{R})}\norm b_{H^{s_0}(\mathbb{R})}.
\end{equation*}
Seen that $\forall \varepsilon > 0$, there exists $C_{\varepsilon} >0$ such that for all $q\in \mathbb{R}^+$:
$\sqrt{q}2^{-q\varepsilon}\leq C_{\varepsilon}$, we infer that:
$$\norm { T_a^vb}_{H^{\frac{1}{2}}(\mathbb{R})} \lesssim \norm a_{H^{\frac{1}{2}}(\mathbb{R})}\norm b_{H^{s_0}(\mathbb{R})},$$
and \eqref{T_ab} follows for all $s\leq \frac{1}{2}< s_0$. 
Moreover, by using the embedding $H^{s_0}(\mathbb{R})\hookrightarrow L^{\infty}(\mathbb{R})$ together with the estimate:  
$$
\norm {\Delta^v_q( T_b^va)}_{L^2(\mathbb{R})}\lesssim \norm {S^v_{q-1}b}_{L^{\infty}(\mathbb{R})} \norm {\Delta_q^v a}_{L^2(\mathbb{R})} \lesssim  \norm b_{L^{\infty}(\mathbb{R})}c_q 2^{-qs} \norm a_{H^s(\mathbb{R})},
$$
we obtain:$$
\norm { T_b^va}_{H^s(\mathbb{R})} \lesssim \norm a_{H^s(\mathbb{R})}\norm b_{H^{s_0}(\mathbb{R})}.$$
For the reminder term, if $s+s_0>0$, then applying Lemma \ref{ber} together with \eqref{stat2} gives:   
\begin{align*}
\norm {\Delta^v_q( R(a,b))}_{L^2(\mathbb{R})} &\lesssim 2^{\frac{q}{2}} \sum_{j\geq q-N_0} \norm {\Delta_j^v a}_{L^2(\mathbb{R})}\norm {\tilde{\Delta}_j^v b}_{L^2(\mathbb{R})},  \\ 
&\lesssim 2^{\frac{q}{2}} \sum_{j\geq q-N_0}\big(c_j^2 2^{-j(s+s_0)} \big) \norm {a}_{H^s(\mathbb{R})}\norm {b}_{H^{s_0}(\mathbb{R})},
\end{align*}
where $$  \tilde{\Delta}_q^v\bydef \sum_{i=\{-1,0,1\}} \Delta_{q+i}^v.$$
Consequently, for any $q\geq -1$ we get:
\begin{align*}
2^{qs}\norm {\Delta^v_q( R(a,b))}_{L^2(\mathbb{R})} &\lesssim  2^{-q(s_0-\frac{1}{2})} \norm {a}_{H^s(\mathbb{R})}\norm {b}_{H^{s_0}(\mathbb{R})} \sum_{j\geq q-N_0}c_j^2 2^{-(j-q)(s+s_0)} \\
&\lesssim c_q 2^{-q(s_0-\frac{1}{2})} \norm {a}_{H^s(\mathbb{R})}\norm {b}_{H^{s_0}(\mathbb{R})}.
\end{align*}
It is then easy to show that $\norm { R(a,b)}_{H^s(\mathbb{R})} \lesssim \norm a_{H^s(\mathbb{R})}\norm b_{H^{s_0}(\mathbb{R})}$.\\
If $s+s_0=0$, then along the same lines we can prove that:
\begin{equation*}
\norm { R(a,b)}_{B^{-\frac{1}{2}}_{2,\infty}(\mathbb{R})} \lesssim \norm a_{H^s(\mathbb{R})}\norm b_{H^{s_0}(\mathbb{R})}.
\end{equation*}
The last step consists to use the following inequality by taking $a=-\frac{1}{2}$ and $\varepsilon=s_0-\frac{1}{2}$:
\begin{equation}\label{nonhomogeneous-embedding}
\norm f_{B^{a-\varepsilon}_{2,1}}=\sum_{k\geq -1} 2^{k(a-\varepsilon)} \norm {\Delta_k f}_{L^2} \leq C(\varepsilon) \norm f_{B^{a}_{2,\infty}}, \;\;  \forall a\in\mathbb{R},\varepsilon>0.
\end{equation}
We get:  
\begin{align*}
\norm { R(a,b)}_{H^{-s_0}(\mathbb{R})}&=\norm { R(a,b)}_{H^{s}(\mathbb{R})} \lesssim \norm a_{H^s(\mathbb{R})}\norm b_{H^{s_0}(\mathbb{R})},
\end{align*}
which ends the proof.
\end{proof}

\noindent Another important result is the following commutator-type estimate:
\begin{lemma}\label{commutator.lemma}
 Let $u,f$ be regular where $u$ is a divergence free vector field in $\mathbb{R}^3$. We have:
$$
\left\|\big[ \Delta_q^v, S_{j-1}^vu^3(.,x_3) \big]f \right\|_{L^2_{v}H^{-\frac{1}{2}}_h}\lesssim 2^{-q}\left\| S_{j-1}^v\nabla_h u(.,x_3) \right\|_{L^\infty_vL^{2}_h} \left\|f\right\|_{L^2_vH^{\frac{1}{2}}_h}.
$$
\end{lemma}
\begin{proof}
The proof is essentially based on the fact that $-\partial_3 u^3 = \nabla_h \cdot u^h$ and the following usual commutator estimate used with respect to the vertical variable:
\begin{equation}\label{commutator}
\norm {\big[ \Delta_j, a \big]b }_{L^r}\lesssim 2^{-j} \norm{\nabla a}_{L^p} \norm{b}_{L^q}, \;\; \text{with} \; \frac{1}{r}= \frac{1}{p} + \frac{1}{q}.
\end{equation}
For the proof of estimates of type \eqref{commutator} one may see for example \cite{Chemin1}, and for a detailed proof of Lemma \ref{commutator.lemma}, one may see \cite{Chemin}.
\end{proof}

\noindent Let us end this section by recalling the Osgood's Lemma (see for instance \cite{Chemin1}):
\begin{lemma}\textbf{Osgood's lemma}\\
Let $g$ be a measurable function from $[t_0,T]$ to $[0,a]$, $\gamma $ a locally integrable function from $[t_0,T]$ to $\mathbb{R}^+$ and $\mu$ a continuous and non-decreasing function from $[0,a]$ to $\mathbb{R}^+$. Assume that for some non-negative real number $c$, $g$ satisfies:
$$
g(t) \leq c + \int_{t_0}^t \gamma(\tau) \mu (g(\tau)) d\tau, \; \; a.e.\; \; t\in [t_0,T].
$$
Then we have for a.a. $t\in [t_0,T]$: 
\begin{align*}
&c>0 \ \Longrightarrow \
-M(g(t))+ M(c) \leq \int_{t_0}^t \gamma(\tau) d\tau,\quad \text{where } M(x)= \int_{x}^a \frac{d\tau}{\mu(\tau)}. \\
&c=0 \text{ and }\int_0^a \frac{d\tau}{\mu(\tau)} = \infty \ \Longrightarrow \ g=0.
\end{align*}
\end{lemma}
\section{A priori estimates for the uniqueness topic}

\noindent In this section we establish the main a priori estimates required to prove the uniqueness in our theorems.

\noindent Let $\langle f,g \rangle \bydef\langle f,g \rangle_{L^2(\mathbb{R}^3)}$ be the usual $L^2$-scalar product, and $\langle f,g \rangle_{\alpha,\beta}$ denotes the scalar product between $f$ and $g$ in $H^{\alpha,\beta}(\mathbb{R}^3)$. In order to simplify the redaction, we introduce the following notations:
\begin{alignat*}{2}
&L_1\bydef \displaystyle\sum_{q\geq -1}2^{2q(s-1)}\langle \Delta_q^v(u^h\cdot\nabla_hw),\Delta_q^vw \rangle, \quad 
&&L_2\bydef \displaystyle\sum_{q\geq -1}2^{2q(s-1)}\langle \Delta_q^v(u^3\partial_3 w),\Delta_q^vw \rangle, \\ 
&L_3\bydef \displaystyle\sum_{q\geq -1}2^{2q(s-1)}\langle \Delta_q^v(w^h\cdot\nabla_hv),\Delta_q^vw \rangle, \quad 
&&L_4\bydef \displaystyle\sum_{q\geq -1}2^{2q(s-1)}\langle \Delta_q^v(w^3\partial_3 v),\Delta_q^vw \rangle, \\ 
&L_5\bydef \displaystyle\sum_{q\geq -1}2^{-2qs}\langle \Delta_q^v(u^h\cdot\nabla_h\theta),\Delta_q^v\theta \rangle, \quad 
&&L_6\bydef \displaystyle\sum_{q\geq -1}2^{-2qs}\langle \Delta_q^v(u^3\partial_3\theta),\Delta_q^v\theta \rangle, \\ 
&L_7\bydef \displaystyle\sum_{q\geq -1}2^{-2qs}\langle \Delta_q^v(w^h\cdot\nabla_h\eta),\Delta_q^v\theta \rangle, \quad 
&&L_8\bydef\displaystyle\sum_{q\geq -1}2^{-2qs}\langle \Delta_q^v(w^3\partial_3 \eta),\Delta_q^v\theta \rangle, \\ 
&L_9\bydef\displaystyle\sum_{q\geq -1}2^{2q(s-1) }\langle \Delta_q^v\theta, \Delta_q^v(w^3) \rangle. &&
\end{alignat*}
We shall prove: 
\begin{prop} \label{proposition1}
Let $s\in]\frac{1}{2},1]$. Then for $u,v,w,\rho,\eta,\theta$ verifying:  
\begin{alignat*}{2}
&u,v,\nabla_h u, \nabla_h v \in H^{0,s}, \quad &&\rho,\eta,\nabla_h \rho, \nabla_h \eta \in H^{0,1-s}, \\
&w,\nabla_h w \in H^{0,s-1}, \quad &&\theta,\nabla_h \theta \in H^{0,-s}, \\ 
&\div u=\div v=\div w=0, &&
\end{alignat*}
we have:
\begin{alignat*}{2}
&L_1\lesssim \left\|u \right\|_{\frac{1}{2},s} \left\|\nabla_h w  \right\|_{0,s-1} \left\| w \right\|_{\frac{1}{2},s-1}, \quad 
&&L_2\lesssim\left\|\nabla_h u \right\|_{0,s} \left\| w \right\|_{\frac{1}{2},s-1}^2, \\
&L_3\lesssim\left\| \nabla_h v \right\|_{0,s} \left\| w \right\|_{\frac{1}{2},s-1}^2, \quad 
&&L_4\lesssim\left\|v \right\|_{\frac{1}{2},s} \big(\left\|w\right\|_{0,s-1} +  \left\|\nabla_h w\right\|_{0,s-1} \big) \left\| w \right\|_{\frac{1}{2},s-1}, \\ 
&L_5\lesssim \left\| u\right\|_{\frac{1}{2},s} \left\|\nabla_h \theta \right\|_{0,-s} \left\|\theta \right\|_{\frac{1}{2},-s}, 
\quad &&L_6\lesssim \left\|\nabla_h u \right\|_{0,s} \left\| \theta\right\|_{\frac{1}{2},-s}^2, \\ 
&L_7\lesssim \left\|\nabla_h \eta\right\|_{0,1-s} \left\|w \right\|_{\frac{1}{2},s-1} \left\| \theta\right\|_{\frac{1}{2},-s}, 
\quad &&L_8\lesssim \left\| \eta\right\|_{\frac{1}{2},1-s} \big(\left\|w \right\|_{0,s-1} + \left\|\nabla_h w \right\|_{0,s-1} \big)\left\| \theta\right\|_{\frac{1}{2},-s}, \\ 
&L_9\lesssim \norm {\theta}_{0,-s}(\norm {\nabla_h w}_{0,s-1}+ \norm {w}_{0,s-1} ). &&
\end{alignat*}
\end{prop}

\begin{proof}
In the following we denote by $c_q$ some constant $c_q \bydef c_q(u,v,w,\theta,\rho,t)$ with $\displaystyle\sum_{q\geq -1}c_q^2 \leq 1$ which comes from the fact (\ref{stat}) or \eqref{stat2}. This constant is allowed to differ from one line to another. \\ 

\noindent $\bullet$~\textbf{$L_1$ estimate }\\
Since $s+(s-1)>0$, by using product Lemma \ref{productrule1} between $H^{\frac{1}{2},s}$ and $H^{0,s-1}$, we obtain
\begin{align*}
L_1= \langle u^h\cdot\nabla_hw , w \rangle_{0,s-1} &\leq \left\| u^h\cdot\nabla_hw \right\|_{-\frac{1}{2},s-1} \left\| w \right\|_{\frac{1}{2},s-1}\\
  & \lesssim \left\|u \right\|_{\frac{1}{2},s} \left\|\nabla_h w \right\|_{0,s-1} \left\| w \right\|_{\frac{1}{2},s-1}.
\end{align*}
\noindent $\bullet$~\textbf{$L_2$ estimate }\\
We write $L_2=L^{(1)}_2+L^{(2)}_2+L^{(3)}_2$ where
\begin{align*}
L^{(1)}_2&\bydef \displaystyle\sum_{q\geq -1}2^{2q(s-1)}\displaystyle\sum_{|j-q|\leq N_0}\langle\Delta_q^v\big( \Delta_j^v(u^3)S_{j-1}^v(\partial_3w)\big),\Delta_q^v w\rangle,\\
L^{(2)}_2&\bydef \displaystyle\sum_{q\geq -1}2^{2q(s-1)}\displaystyle\sum_{|j-q|\leq N_0}\langle\Delta_q^v\big( S_{j-1}^v(u^3)\Delta_j^v(\partial_3w)\big),\Delta_q^v w\rangle,\\
L^{(3)}_2&\bydef \displaystyle\sum_{q\geq -1}2^{2q(s-1)}\displaystyle\sum_{i\in\{0,-1,1\}}\displaystyle\sum_{j\geq q-N_0}\langle\Delta_q^v\big( \Delta_{j+1}^v(u^3)\Delta_j^v(\partial_3w)\big),\Delta_q^v w\rangle.
\end{align*}
Then, by using the embedding of $H^{\frac{1}{2}}(\mathbb{R}_h^2)$ in $L^4(\mathbb{R}_h^2)$, Bernstein Lemma for the vertical variable together with statement \eqref{stat}, we obtain
\begin{align*}
L^{(1)}_2 &\lesssim \left\|w \right\|_{\frac{1}{2},s-1} \displaystyle\sum_{q\geq -1}c_q2^{q(s-1)}2^{q/2}\displaystyle\sum_{|j-q|\leq N_0}\left\|S_{j-1}^vw\right\|_{L^4_hL^{2}_v} \left\|\Delta_j^v\nabla_h u\right\|_{L^2}\\
&\lesssim \left\|w \right\|_{\frac{1}{2},s-1}\left\|w\right\|_{\frac{1}{2},s-1}\left\|\nabla_h u\right\|_{0,s} \displaystyle\sum_{q\geq -1}c_q2^{q(s-\frac{1}{2})}\displaystyle\sum_{|j-q|\leq N_0}c_j^22^{j(1-s)} 2^{-sj}\\
&\lesssim \left\|w \right\|_{\frac{1}{2},s-1}\left\|w\right\|_{\frac{1}{2},s-1}\left\|\nabla_h u\right\|_{0,s} \displaystyle\sum_{q\geq -1}c_q2^{q(s-\frac{1}{2})}2^{q(1-s)} 2^{-qs}\\
&\lesssim \left\|w \right\|_{\frac{1}{2},s-1}\left\|w\right\|_{\frac{1}{2},s-1}\left\|\nabla_h u\right\|_{0,s} \displaystyle\sum_{q\geq -1}c_q2^{-q(s-\frac{1}{2})}\\
&\lesssim \left\|w \right\|_{\frac{1}{2},s-1}\left\|w\right\|_{\frac{1}{2},s-1}\left\|\nabla_h u\right\|_{0,s}.
\end{align*}
To estimate $L^{(2)}_2$, we consider the decomposition used in \cite{Chemin}, by writing $L^{(2)}_2 = A_1+A_2+A_3$, with
\begin{align*}
A_1 &\bydef\displaystyle\sum_{q\geq -1}2^{2q(s-1)} \langle S_q^v(u^3)\Delta_q^v(\partial_3w),\Delta_q^v w\rangle,\\
A_2&\bydef\displaystyle\sum_{q\geq -1}2^{2q(s-1)}\displaystyle\sum_{|j-q|\leq N_0}\langle\big(S_q^v - S_{j-1}^v\big)(u^3)\Delta_j^v(\partial_3w),\Delta_q^v w\rangle,\\
A_3&\bydef \displaystyle\sum_{q\geq -1}2^{2q(s-1)}\displaystyle\sum_{|j-q|\leq N_0}\langle\big[\Delta_q^v, S_{j-1}^v(u^3)\big]\Delta_j^v(\partial_3w),\Delta_q^vw\rangle,
\end{align*}
where $[\Delta_q^v, S_{j-1}^v(u^3)\big]$ denotes the commutator between $\Delta_q^v$ and $S_{j-1}^v(u^3)$.\\
After integration by parts we obtain
\begin{align*}
A_1&=- \frac{1}{2} \displaystyle\sum_{q\geq -1}2^{2q(s-1)} \langle S_q^v(\partial_3u^3)\Delta_q^v(w),\Delta_q^v w\rangle\\
&=\frac{1}{2} \displaystyle\sum_{q\geq -1}2^{2q(s-1)} \langle S_q^v(\nabla_h\cdot u^h)\Delta_q^v(w),\Delta_q^v w\rangle\\ 
&\lesssim \left\|w \right\|_{\frac{1}{2},s-1} \displaystyle\sum_{q\geq -1}c_q2^{-q(s-1)}\left\|\Delta_{q}^vw\right\|_{L^4_hL^{2}_v} \left\|S_q^v(\nabla_h u)\right\|_{L^2_hL^{\infty}_v}\\
&\lesssim \left\|w \right\|_{\frac{1}{2},s-1}^2\left\|\nabla_h u\right\|_{0,s}.
\end{align*}
In order to estimate $A_2$, we remark first that $S_q^v-S_{j-1}^v$ is supported away from $0$ in Fourier side, that is we can use Lemma \ref{ber} to estimate $A_2$ just like $L_2^{(1)}$. Indeed
\begin{align*}
A_2 &\lesssim \left\|w \right\|_{\frac{1}{2},s-1} \displaystyle\sum_{q\geq -1}c_q2^{q(s-1)}\displaystyle\sum_{|j-q|\leq N_0}\left\|\big(S_q^v-S_{j-1}^v\big)\partial_3 u^3\right\|_{L^2_hL^{\infty}_v} \left\|\Delta_j^vw\right\|_{L_h^4L^2_v}\\
&\lesssim \left\|w \right\|_{\frac{1}{2},s-1} \displaystyle\sum_{q\geq -1}c_q2^{q(s-1)}\displaystyle\sum_{|j-q|\leq N_0}\left\|\big(S_q^v-S_{j-1}^v\big)\nabla u^h\right\|_{L^2_hL^{\infty}_v}\left\|\Delta_j^vw\right\|_{L^4_hL^2_v}\\
&\lesssim \left\|w \right\|_{\frac{1}{2},s-1}^2\left\|\nabla_h u \right\|_{0,s} \displaystyle\sum_{|i|\leq N_0}\displaystyle\sum_{q\geq -1}c_q2^{q(s-1)}c_{j+i}2^{(q+i)(1-s)}\\
&\lesssim \left\|w \right\|_{\frac{1}{2},s-1}^2\left\|\nabla_h u \right\|_{0,s} .
\end{align*}
Finally, for $A_3$ we use the commutator estimate proved in Lemma \ref{commutator.lemma} to obtain
\begin{align*}
A_3 &\lesssim \left\|w \right\|_{\frac{1}{2},s-1} \displaystyle\sum_{q\geq -1}c_q2^{q(s-1)}\displaystyle\sum_{|j-q|\leq N_0}\left\|\big[\Delta_q^v, S_{j-1}^v(u^3)\big]\Delta_j^v(\partial_3w)\right\|_{L^2\big(\mathbb{R}_{x_3};H^{-\frac{1}{2}}(\mathbb{R}^2)\big)}\\
&\lesssim \left\|w \right\|_{\frac{1}{2},s-1} \displaystyle\sum_{q\geq -1}c_q2^{q(s-1)}\displaystyle\sum_{|j-q|\leq N_0}\left\| S_{j-1}^v\nabla_h u(.,x_3) \right\|_{L^2_hL^{\infty}_v} \left\|\Delta_j^v w\right\|_{\frac{1}{2},0}\\
&\lesssim \left\|w \right\|_{\frac{1}{2},s-1}^2 \left\|\nabla_h u\right\|_{0,s}\displaystyle\sum_{i\in\{0,-1,1\}}\displaystyle\sum_{q\geq -1}c_qc_{q+i}\\
&\lesssim \left\|w \right\|_{\frac{1}{2},s-1}^2 \left\|\nabla_h u\right\|_{0,s}.
\end{align*}
Ditto for the last term in this part, using the fact that $\partial_3u^3=-\nabla_h\cdot u^h$, it happens
\begin{align*}
L^{(3)}_2 &\lesssim \left\|w \right\|_{\frac{1}{2},s-1} \displaystyle\sum_{q\geq -1}c_q2^{q(s-1)}2^{q/2}\displaystyle\sum_{i\in\{0,-1,1\}}\displaystyle\sum_{j\geq q-N_0}\left\|\Delta_{j+i}^vw\right\|_{L^4_hL^{2}_v} \left\|\Delta_j^v\nabla_h u\right\|_{L^2}\\
&\lesssim \left\|w \right\|_{\frac{1}{2},s-1}\left\|w\right\|_{\frac{1}{2},s-1}\left\|\nabla_h u\right\|_{0,s} \displaystyle\sum_{q\geq -1}c_q2^{q(s-\frac{1}{2})}\displaystyle\sum_{i\in\{0,-1,1\}}\displaystyle\sum_{j\geq q-N_0}c_jc_{j+i}2^{j(1-s)} 2^{-sj}\\
&\lesssim \left\|w \right\|_{\frac{1}{2},s-1}^2\left\|\nabla_h u\right\|_{0,s} \displaystyle\sum_{q\geq -1}c_q2^{q(s-\frac{1}{2})}2^{q(1-2s)}\displaystyle\sum_{i\in\{0,-1,1\}}\displaystyle\sum_{j\geq q-N_0}c_jc_{j+i},
\end{align*}
where we used the fact that $s\in]\frac{1}{2},1]$ that is $1-2s<0$. We obtain finally
\begin{align*}
L^{(3)}_2 &\lesssim \left\|w \right\|_{\frac{1}{2},s-1}^2\left\|\nabla_h u\right\|_{0,s} \displaystyle\sum_{q\geq -1}c_q2^{-q(s-\frac{1}{2})}\\
&\lesssim \left\|w \right\|_{\frac{1}{2},s-1}^2\left\|\nabla_h u\right\|_{0,s}.
\end{align*}
\textbf{Remark} In the case where $s=1$ we do not have to deal with $L_1+L_2$ which is equal to $0$ because of the identity $\langle u\cdot \nabla w,w \rangle=0$.\\
\noindent $\bullet$~\textbf{$L_3$ estimate }\\
By using product Lemma \ref{productrule1} between $ H^{0,s}$ and $ H^{\frac{1}{2},s-1}$ we obtain
\begin{align*}
L_3= \langle w^h\nabla_hv , w \rangle_{0,s-1} &\leq \left\| w^h\nabla_hv \right\|_{-\frac{1}{2},s-1} \left\| w \right\|_{\frac{1}{2},s-1}\\
  & \lesssim \left\|\nabla_hv\right\|_{0,s} \left\|w \right\|_{\frac{1}{2},s-1}^2.
\end{align*}
\noindent $\bullet$~\textbf{$L_4$ estimate }\\
We write $L_4=L^{(1)}_4+L^{(2)}_4$, where
\begin{align*}
L^{(1)}_4 \bydef\displaystyle\sum_{q\geq -1}2^{2q(s-1)}\displaystyle\sum_{j\geq q-N_0}\langle \Delta_q^v(\Delta_j^vw^3S_{j+2}^v(\partial_3 v)),\Delta_q^vw \rangle ,
\end{align*}
\begin{align*}
L^{(2)}_4\bydef \displaystyle\sum_{q\geq -1}2^{2q(s-1)}\displaystyle\sum_{|j-q|\leq N_0}\langle \Delta_q^v(S_{j-1}^v(w^3)\Delta_j^v\partial_3 v),\Delta_q^vw \rangle.
\end{align*}
Hence, by using again Lemma \ref{productrule1}, we infer that
\begin{align*}
L^{(1)}_4 &\leq \left\|w \right\|_{\frac{1}{2},s-1}\displaystyle\sum_{q\geq -1}c_q 2^{q(s-\frac{1}{2})}\displaystyle\sum_{j\geq q-N_0}\left\|\Delta_j^vw^3 \right\|_{L^2} \left\| S_{j+2}^v(\partial_3 v)\right\|_{L^4_hL^2_v}.
\end{align*}
For $s\neq 1$, after certain calculations, we get
\begin{align*}
\left\| S_{j+2}^v(\partial_3 v)\right\|_{L^4_hL^2_v}&\leq  \displaystyle\sum_{m\leq j+1} 2^{m(1-s)}2^{sm}\left\|\Delta_m^vv\right\|_{L^4_hL^2_v}\\
&\leq \bigg(\displaystyle\sum_{m\leq j+1} 2^{2m(1-s)}\bigg)^{\frac{1}{2}}\left\|v \right\|_{\frac{1}{2},s} \\
&\lesssim 2^{j(1-s)}\left\|v \right\|_{\frac{1}{2},s}.
\end{align*}
Thus, by using Lemma \ref{ber} together with the previous estimate and the divergence free condition on $w$, we find
\begin{align*}
L^{(1)}_4 &\lesssim \left\|w \right\|_{\frac{1}{2},s-1}\left\|v \right\|_{\frac{1}{2},s}\displaystyle\sum_{q\geq -1}c_q 2^{q(s-\frac{1}{2})}\displaystyle\sum_{j\geq q-N_0}2^{-js}\left\|\Delta_j^v\nabla_h w \right\|_{L^2}\\
&\lesssim \left\|w \right\|_{\frac{1}{2},s-1}\left\|v \right\|_{\frac{1}{2},s}\left\|\nabla_h w\right\|_{0,s-1}\displaystyle\sum_{q\geq -1}c_q 2^{q(s-\frac{1}{2})}\displaystyle\sum_{j\geq q-N_0}c_j2^{j(1-2s)} \\
&\lesssim \left\|w \right\|_{\frac{1}{2},s-1}\left\|v \right\|_{\frac{1}{2},s}\left\|\nabla_h w\right\|_{0,s-1}\displaystyle\sum_{q\geq -1}c_q 2^{q(s-\frac{1}{2})}\bigg(\displaystyle\sum_{j\geq q-N_0}2^{2j(1-2s)}\bigg)^{\frac{1}{2}}\left\| c_j \right\|_{l^2(\mathbb{N}\cup\{-1\})}\\
&\lesssim \left\|w \right\|_{\frac{1}{2},s-1}\left\|v \right\|_{\frac{1}{2},s}\left\|\nabla_h w\right\|_{0,s-1}\displaystyle\sum_{q\geq -1}c_q 2^{q(s-\frac{1}{2})}2^{q(1-2s)}\\
&\lesssim \left\|w \right\|_{\frac{1}{2},s-1}\left\|v \right\|_{\frac{1}{2},s}\left\|\nabla_h w\right\|_{0,s-1}\displaystyle\sum_{q\geq -1}c_q 2^{q(\frac{1}{2}-s)}\\
&\lesssim \left\|w \right\|_{\frac{1}{2},s-1}\left\|v \right\|_{\frac{1}{2},s}\left\|\nabla_h w\right\|_{0,s-1}.
\end{align*}
For the second term we proceed as follows
\begin{align*}
L^{(2)}_4 &\lesssim \left\|w \right\|_{\frac{1}{2},s-1} \displaystyle\sum_{q\geq -1}c_q2^{q(s-1)}2^q\displaystyle\sum_{|j-q|\leq N_0}2^{j-q}\left\|S_{j-1}^vw^3\right\|_{L^2_hL^{\infty}_v} \left\|\Delta_j^vv\right\|_{L^4_hL^{2}_v} \\
&\lesssim \left\|w \right\|_{\frac{1}{2},s-1}\left\|w^3\right\|_{0,s}\displaystyle\sum_{|j|\leq N_0} \displaystyle\sum_{q\geq -1}c_q2^{q(s-1)}2^q\left\|\Delta_{j+q}^vv\right\|_{L^4_hL^{2}_v}\\
&\lesssim\left\|v \right\|_{\frac{1}{2},s} \left\|w \right\|_{\frac{1}{2},s-1}\left\|w^3\right\|_{0,s}\displaystyle\sum_{|i|\leq N_0} \displaystyle\sum_{q\geq -1}c_qc_{q+i}\\
&\lesssim\left\|v \right\|_{\frac{1}{2},s} \left\|w \right\|_{\frac{1}{2},s-1}\left\|w^3\right\|_{0,s}.
\end{align*}
In order to close the estimates of $L^{(2)}_4$ we remark that, for any $s\in [\frac{1}{2},1]$, we have
\begin{align*}
\left\|w^3\right\|_{0,s}&\leq \left\|w^3\right\|_{0,s-1} +\left\|\partial_3w^3\right\|_{0,s-1}\\
&\leq \left\|w^3\right\|_{0,s-1} +\left\|\nabla_hw\right\|_{0,s-1}.
\end{align*}

\noindent In the case where $s=1$, note that the estimate can be obtained easily, by using product rules and the previous inequality, as the following
\begin{align*}
\langle w^3\partial_3 v , w \rangle_{L^2} &\leq \norm {w^3\partial_3v}_{-\frac{1}{2},0}\norm {w}_{\frac{1}{2},0}\\
&\lesssim \norm {w^3}_{0,1}\norm {\partial_3 v}_{\frac{1}{2},0}\norm {w}_{\frac{1}{2},0}\\
&\lesssim \norm {v}_{\frac{1}{2},1}(\norm {w}_{L^2}+\norm {\nabla_hw}_{L^2})\norm {w}_{\frac{1}{2},0}.
\end{align*}
\noindent $\bullet$~\textbf{$L_5$ estimate }\\
In order to estimate this term we proceed by duality by inferring firstly that
\begin{equation*}
L_5\leq \norm {u^h\nabla_h \theta}_{-\frac{1}{2}, -s} \norm \theta_{\frac{1}{2},-s}.
\end{equation*}
Moreover, Lemma \ref{productrule1} gives
 \begin{equation*}
 \norm {u^h\nabla_h \theta}_{-\frac{1}{2}, -s}\lesssim \norm {u^h}_{\frac{1}{2},s} \norm {\nabla_h\theta}_{0,-s} .
 \end{equation*}
It follows that
 \begin{equation*}
 L_5\lesssim \norm {u^h}_{\frac{1}{2},s} \norm {\nabla_h\theta}_{0,-s} \norm \theta_{\frac{1}{2},-s}.
 \end{equation*}
\noindent $\bullet$~\textbf{$L_6$ estimate} \\ 
We use the Bony's decomposition $L_6= L_6^{(1)}+L_6^{(2)}+L_6^{(3)}$, where
\begin{align*}
L^{(1)}_6&\bydef \displaystyle\sum_{q\geq -1}2^{2q(-s)}\displaystyle\sum_{|j-q|\leq N_0}\langle\Delta_q^v\big( \Delta_j^v(u^3)S_{j-1}^v(\partial_3\theta)\big),\Delta_q^v \theta\rangle,\\
L^{(2)}_6&\bydef \displaystyle\sum_{q\geq -1}2^{2q(-s)}\displaystyle\sum_{|j-q|\leq N_0}\langle\Delta_q^v\big( S_{j-1}^v(u^3)\Delta_j^v(\partial_3\theta)\big),\Delta_q^v \theta\rangle,\\
L^{(3)}_6&\bydef \displaystyle\sum_{q\geq -1}2^{2q(-s)}\displaystyle\sum_{i\in\{0,-1,1\}}\displaystyle\sum_{j\geq q-N_0}\langle\Delta_q^v\big( \Delta_{j+1}^v(u^3)\Delta_j^v(\partial_3\theta)\big),\Delta_q ^v\theta\rangle.
\end{align*} 
For the first term, we use the Bernstein Lemma together with usual Sobolev embedding and the free divergence condition to obtain
\begin{align*}
L^{(1)}_6 &\leq  C \left\|\theta \right\|_{\frac{1}{2},-s} \displaystyle\sum_{q\geq -1}c_q2^{-qs}2^{q/2}\displaystyle\sum_{|j-q|\leq N_0}\left\|S_{j-1}^v\theta\right\|_{L^4_hL^{2}_v} \left\|\Delta_j^v\nabla_h \cdot u^h\right\|_{L^2}\\
&\lesssim \left\|\theta \right\|_{\frac{1}{2},-s}\left\|\theta\right\|_{\frac{1}{2},s-1}\left\|\nabla_h u\right\|_{0,s} \displaystyle\sum_{q\geq -1}c_q2^{q(\frac{1}{2}-s)}\displaystyle\sum_{|j-q|\leq N_0}c_j^2\\
&\lesssim \left\|\theta \right\|_{\frac{1}{2},-s}^2\left\|\nabla_h u\right\|_{0,s}.
\end{align*}
For $L^{(2)}_6$, we follow the same decomposition used for $L^{(2)}_2 $, so we write $L^{(2)}_6 = B_1+B_2+B_3$, 
where
\begin{align*}
B_1 &\bydef\displaystyle\sum_{q\geq -1}2^{-2qs} \langle S_q^v(u^3)\Delta_q^v(\partial_3\theta),\Delta_q^v \theta\rangle,\\
B_2&\bydef\displaystyle\sum_{q\geq -1}2^{-2qs}\displaystyle\sum_{|j-q|\leq N_0}\langle\big(S_q^v - S_{j-1}^v\big)(u^3)\Delta_j^v(\partial_3\theta),\Delta_q^v \theta\rangle,\\
B_3&\bydef \displaystyle\sum_{q\geq -1}2^{-2qs}\displaystyle\sum_{|j-q|\leq N_0}\langle\big[\Delta_q^v, S_{j-1}^v(u^3)\big]\Delta_j^v(\partial_3\theta),\Delta_q^v \theta\rangle.
\end{align*}
After integration by parts, we obtain
\begin{align*}
B_1&= \frac{1}{2}\displaystyle\sum_{q\geq -1}2^{-2qs} \langle S_q^v(\partial_3u^3)\Delta_q^v(\theta),\Delta_q^v \theta\rangle\\
&= \frac{1}{2}\displaystyle\sum_{q\geq -1}2^{-2qs} \langle S_q^v(\nabla_h u^h)\Delta_q^v(\theta),\Delta_q^v \theta\rangle\\ 
&\lesssim \left\|\theta \right\|_{\frac{1}{2},-s} \displaystyle\sum_{q\geq -1}c_q2^{-qs}\left\|\Delta_{q}^v\theta\right\|_{L^4_hL^{2}_v} \left\|S_q^v(\nabla_h u)\right\|_{L^2_hL^{\infty}_v}\\
&\lesssim \left\|\theta \right\|_{\frac{1}{2},s-1}^2\left\|\nabla_h u\right\|_{0,s}.
\end{align*}
To estimate $B_2$ we remark first that $S_q^v-S_{j-1}^v$ is supported away from $0$ in Fourier side, that is we can use Lemma \ref{ber} and estimate $B_2$ as $A_2$, indeed
\begin{align*}
B_2 &\lesssim \left\|\theta \right\|_{\frac{1}{2},-s} \displaystyle\sum_{q\geq -1}c_q2^{-qs}\displaystyle\sum_{|j-q|\leq N_0}\left\|\big(S_q^v-S_{j-1}^v\big)\partial_3 u^3\right\|_{L^2_hL^{\infty}_v} \left\|\Delta_j^v\theta\right\|_{L^4_hL^2_v}\\
&\lesssim \left\|\theta \right\|_{\frac{1}{2},-s} \displaystyle\sum_{q\geq -1}c_q2^{-qs}\displaystyle\sum_{|j-q|\leq N_0}\left\|\big(S_q^v-S_{j-1}^v\big)\nabla_h u^h\right\|_{L^2_hL^{\infty}_v}\left\|\Delta_j^v\theta\right\|_{L^4_hL^2_v}\\
&\lesssim \left\|\theta \right\|_{\frac{1}{2},-s}^2\left\|\nabla_h u \right\|_{0,s} \displaystyle\sum_{|i|\leq N_0}\displaystyle\sum_{q\geq -1}c_qc_{q+i}\\
&\lesssim \left\|\theta \right\|_{\frac{1}{2},-s}^2\left\|\nabla_h u \right\|_{0,s} .
\end{align*}
Finally, for $B_3$ we use the commutator estimate proved in Lemma \ref{commutator.lemma} to obtain
\begin{align*}
B_3 &\lesssim \left\|\theta \right\|_{\frac{1}{2},-s} \displaystyle\sum_{q\geq -1}c_q2^{-qs}\displaystyle\sum_{|j-q|\leq N_0}\left\|\big[\Delta_q^v, S_{j-1}^v(u^3)\big]\Delta_j^v(\partial_3\theta)\right\|_{L^2\big(\mathbb{R}_{x_3};H^{-\frac{1}{2}}(\mathbb{R}^2)\big)}\\
&\lesssim  \left\|\theta \right\|_{\frac{1}{2},-s} \displaystyle\sum_{q\geq -1}c_q2^{-qs}\displaystyle\sum_{|j-q|\leq N_0}\left\| S_{j-1}^v\nabla_h u(.,x_3) \right\|_{L^2_hL^{\infty}_v} \left\|\Delta_j ^v\theta\right\|_{\frac{1}{2},0}\\
&\lesssim \left\|\theta \right\|_{\frac{1}{2},-s}^2 \left\|\nabla_h u\right\|_{0,s}\displaystyle\sum_{i\in\{0,-1,1\}}\displaystyle\sum_{q\geq -1}c_qc_{q+i}\\
&\lesssim \left\|\theta \right\|_{\frac{1}{2},-s}^2 \left\|\nabla_h u\right\|_{0,s}.
\end{align*}
For $L^{(3)}_6$, by using the same arguments we find
\begin{align*}
L^{(3)}_6 &\lesssim \left\|\theta \right\|_{\frac{1}{2},-s} \displaystyle\sum_{q\geq -1}c_q2^{-qs}2^{q/2}\displaystyle\sum_{i\in\{0,-1,1\}}\displaystyle\sum_{j\geq q-N_0}\left\|\Delta_{j+i}^v\theta\right\|_{L^4_hL^{2}_v} \left\|\Delta_j^v\nabla_h \cdot u^h\right\|_{L^2}\\
&\lesssim \left\|\theta \right\|_{\frac{1}{2},-s}\left\|\theta\right\|_{\frac{1}{2},-s}\left\|\nabla_h u\right\|_{0,s} \displaystyle\sum_{q\geq -1}c_q2^{-q(s-\frac{1}{2})}\displaystyle\sum_{i\in\{0,-1,1\}}\displaystyle\sum_{j\geq q-N_0}c_jc_{j+i}\\
&\lesssim \left\|\theta \right\|_{\frac{1}{2},-s}^2\left\|\nabla_h u\right\|_{0,s}.\\
\end{align*}
\noindent $\bullet$~\textbf{$L_7$ estimate }\\
This term can be estimated by using the following property based on product rules in dimension one together with inequality \eqref{nonhomogeneous-embedding}
$$ H^{s-1}(\mathbb{R}) \cdot  H^{1-s}(\mathbb{R}) \subset B^{-\frac{1}{2}}_{2,\infty}(\mathbb{R}) \hookrightarrow H^{-s}.$$
Indeed, based on the Bony's decomposition with respect to the vertical variable we write $L_7= L_7^{(1)}+L_7^{(2)}+L_7^{(3)}$, 
where
\begin{align*}
L_7^{(1)} &\bydef \displaystyle\sum_{q\geq -1}2^{-2qs}\displaystyle\sum_{|j-q|\leq N_0}\langle\Delta_q^v\big( S_{j-1}^v(w^h)\Delta_j^v(\nabla_h\eta)\big),\Delta_q^v\theta \rangle,\\
L_7^{(2)} &\bydef \displaystyle\sum_{q\geq -1}2^{-2qs}\displaystyle\sum_{|j-q|\leq N_0}\langle \Delta_q^v\big(\Delta_j^v(w^h)S_{j-1}^v(\nabla_h\eta)\big),\Delta_q^v\theta \rangle,\\
L_7^{(3)} &\bydef \displaystyle\sum_{q\geq -1}2^{-2qs}\displaystyle\sum_{i\in\{0,-1,1\}}\displaystyle\sum_{j \geq q-N_0}\langle \Delta_q^v\big( \Delta_j^v(w^h)\Delta_{j+i}^v(\nabla_h\eta)\big),\Delta_q^v\theta \rangle.
\end{align*}
By using inequality \eqref{stat2}, similar arguments give then
\begin{align*}
L_7^{(1)} &\leq \left\| \theta \right\|_{\frac{1}{2},-s} \displaystyle\sum_{q\geq -1}c_q2^{-q(s-\frac{1}{2})}  \displaystyle\sum_{|j-q|\leq N_0}\left\| S_{j-1}^vw \right\|_{L^4_hL^2_v}\left\|\Delta_{j}^v(\nabla_h\eta) \right\|_{L^2}\\
 &\leq \left\|\nabla_h \eta\right\|_{0,1-s} \left\|w \right\|_{\frac{1}{2},s-1} \left\| \theta\right\|_{\frac{1}{2},-s}\displaystyle\sum_{q\geq -1}c_q2^{-q(s-\frac{1}{2})}  \displaystyle\sum_{|j-q|\leq N_0} c_j^2\\
 &\lesssim \left\|\nabla_h \eta\right\|_{0,1-s} \left\|w \right\|_{\frac{1}{2},s-1} \left\| \theta\right\|_{\frac{1}{2},-s}.
\end{align*}
For the second term we proceed as follows
\begin{align*}
L_7^{(2)} &\leq \left\| \theta \right\|_{\frac{1}{2},-s}  \displaystyle\sum_{q\geq -1}c_q2^{-q(s-\frac{1}{2})}  \displaystyle\sum_{|j-q|\leq N_0}2^{\frac{1}{2}(j-q)}\big(\left\|\Delta_{j}^v w \right\|_{L^4_hL^2_v}2^{j(s-1)}\big)\big(2^{j(1-s-\frac{1}{2})}\left\|S_{j-1}^v(\nabla_h\eta) \right\|_{L^2_hL^{\infty}_v}\big)\\
 &\leq \left\|\nabla_h \eta\right\|_{B^{0,1-s-\frac{1}{2}}_{\infty,2}} \left\|w \right\|_{\frac{1}{2},s-1} \left\| \theta\right\|_{\frac{1}{2},-s}\displaystyle\sum_{q\geq -1}c_q2^{-q(s-\frac{1}{2})}  \displaystyle\sum_{|j-q|\leq N_0} c_j^2\\
 &\lesssim \left\|\nabla_h \eta\right\|_{0,1-s} \left\|w \right\|_{\frac{1}{2},s-1} \left\| \theta\right\|_{\frac{1}{2},-s}, 
\end{align*}
where we used the embedding $H^{0,1-s}\hookrightarrow B^{0,1-s-\frac{1}{2}}_{\infty,2}$ and the fact that $1-s-\frac{1}{2}<0$.\\
For the last term we proceed as follows
\begin{align*}
L_7^{(3)} &\leq  \left\| \theta \right\|_{\frac{1}{2},-s} \displaystyle\sum_{q\geq -1}c_q2^{-q(s-\frac{1}{2})}\displaystyle\sum_{i\in\{0,-1,1\}}  \displaystyle\sum_{j\geq q-N_0}\left\| \Delta_j^v w \right\|_{L^4_hL^2_v}\left\|\Delta_{j+i}^v(\nabla_h\eta) \right\|_{L^2}\\
&\leq \left\|\nabla_h \eta\right\|_{0,1-s} \left\|w \right\|_{\frac{1}{2},s-1} \left\| \theta\right\|_{\frac{1}{2},-s}\displaystyle\sum_{i\in\{0,-1,1\}} \displaystyle\sum_{q\geq -1}c_q2^{-q(s-\frac{1}{2})}  \displaystyle\sum_{j\geq q-N_0} c_jc_{j+i}\\
 &\lesssim \left\|\nabla_h \eta\right\|_{0,1-s} \left\|w \right\|_{\frac{1}{2},s-1} \left\| \theta\right\|_{\frac{1}{2},-s}.
\end{align*}
\noindent $\bullet$~\textbf{$L_8$ estimate }\\
We write $L_8=L_8^{(1)}+L_8^{(2)}+L_8^{(3)}$, where
\begin{align*}
L_8^{(1)}\bydef\displaystyle\sum_{q\geq -1}2^{-2qs}\displaystyle\sum_{|j-q|\leq N_0}\langle \Delta_q^v(\Delta_j^vw^3S_{j-1}^v(\partial_3 \eta)),\Delta_q^v\theta \rangle,
\end{align*}
\begin{align*}
L_8^{(2)}\bydef \displaystyle\sum_{q\geq -1}2^{-2qs}\displaystyle\sum_{|j-q|\leq N_0}\langle \Delta_q^v(S_{j-1}^v(w^3)\Delta_j^v\partial_3 \eta),\Delta_q^v \theta\rangle,
\end{align*}
\begin{align*}
L_8^{(3)}\bydef\displaystyle\sum_{q\geq -1}2^{-2qs}\displaystyle\sum_{i\in\{0,-1,1\}}\displaystyle\sum_{j\geq q-N_0}\langle \Delta_q^v(\Delta_j^vw^3\Delta_{j+i}^v(\partial_3 \eta)),\Delta_q^v\theta \rangle .
\end{align*}
Then, for the first term, we have
\begin{align*}
L_8^{(1)}&\leq \left\|\theta \right\|_{\frac{1}{2},-s}\displaystyle\sum_{q\geq -1}c_q 2^{-qs}\displaystyle\sum_{|j-q|\leq N_0}\left\|\Delta_j^vw^3 \right\|_{L^2} \left\| S_{j-1}^v(\partial_3 \eta)\right\|_{L^4_hL^{\infty}_v}.
\end{align*}
Now we use
\begin{align*}
\left\| S_{j-1}^v(\partial_3 \eta)\right\|_{L^4_hL^{\infty}_v}&\leq  2^j c_j 2^{-j(1-s-\frac{1}{2})}\left\|\eta \right\|_{H^{\frac{1}{2}}_h({B^{1-s-\frac{1}{2}}_{\infty,2}})_v}\\
&\lesssim 2^j c_j 2^{-j(1-s-\frac{1}{2})}\left\|\eta \right\|_{H^{\frac{1}{2},1-s}},
\end{align*}
and
\begin{align*}
\left\|\Delta_j^vw^3 \right\|_{L^2}&\leq 2^{-j}c_j2^{-j(s-1)}\left\|\nabla_h w \right\|_{H^{0,s-1}}.
\end{align*}
Therefore we find
\begin{align*}
L_8^{(1)}&\leq \left\|\theta \right\|_{\frac{1}{2},-s}\left\|\nabla_h w \right\|_{H^{0,s-1}}\left\|\eta \right\|_{H^{\frac{1}{2},1-s}}\displaystyle\sum_{q\geq -1}c_q 2^{q(\frac{1}{2}-s)}\displaystyle\sum_{|j-q|\leq N_0}c_j^22^{\frac{1}{2}(j-q)}\\
&\lesssim\left\|\theta \right\|_{\frac{1}{2},-s}\left\|\nabla_h w \right\|_{H^{0,s-1}}\left\|\eta \right\|_{H^{\frac{1}{2},1-s}}.
\end{align*}
Next, for the second term
\begin{align*}
L_8^{(2)}&\leq \left\|\theta \right\|_{\frac{1}{2},-s}\displaystyle\sum_{q\geq -1}c_q \displaystyle\sum_{|j-q|\leq N_0}2^{-(q-j)s}\left\|S_{j-1}^vw^3 \right\|_{L^2_hL^{\infty}_v}2^{j(1-s)}\left\| \Delta_j^v(\eta)\right\|_{L^4_hL^{2}_v}\\
&\leq \left\|\theta \right\|_{\frac{1}{2},-s}\left\|w^3 \right\|_{0,s}\left\|\eta \right\|_{\frac{1}{2},1-s}\displaystyle\sum_{q\geq -1}c_q \displaystyle\sum_{|j-q|\leq N_0}c_j\\
&\leq \left\|\theta \right\|_{\frac{1}{2},-s}\left\|w^3 \right\|_{0,s}\left\|\eta \right\|_{\frac{1}{2},1-s}\displaystyle\sum_{|j|\leq N_0}\displaystyle\sum_{q\geq -1}c_q c_{j+q}\\
&\lesssim \left\|\theta \right\|_{\frac{1}{2},-s}\big(\left\|w^3 \right\|_{0,s-1}+\left\|\nabla_hw \right\|_{0,s-1}\big)\left\|\eta \right\|_{\frac{1}{2},1-s}.
\end{align*}
For $L_8^{(3)}$ we have
\begin{align*}
L_8^{(3)}&\leq \left\|\theta \right\|_{\frac{1}{2},-s}\displaystyle\sum_{q\geq -1}c_q 2^{-qs}2^{q/2}\displaystyle\sum_{i\in{0,-1,1}}\displaystyle\sum_{j\geq q-N_0}\left\|\Delta_j^vw^3 \right\|_{L^2} \left\| \Delta_{j+i}^v(\partial_3 \eta)\right\|_{L^4_hL^{2}_v}\\
&\leq \left\|\theta \right\|_{\frac{1}{2},-s}\displaystyle\sum_{q\geq -1}c_q 2^{q(\frac{1}{2}-s)}\displaystyle\sum_{i\in{0,-1,1}}\displaystyle\sum_{j\geq q-N_0}\left\|\Delta_j^v\nabla_hw \right\|_{L^2} \left\| \Delta_{j+i}^v\eta\right\|_{L^4_hL^{2}_v}\\
&\leq \left\|\theta \right\|_{\frac{1}{2},-s}\displaystyle\sum_{q\geq -1}c_q 2^{q(\frac{1}{2}-s)}\displaystyle\sum_{i\in{0,-1,1}}\displaystyle\sum_{j\geq q-N_0}2^{j(s-1)}\left\|\Delta_j^v\nabla_hw \right\|_{L^2} 2^{j(1-s)} \left\| \Delta_{j+i}^v\eta\right\|_{L^4_hL^{2}_v}\\
&\lesssim \left\|\theta \right\|_{\frac{1}{2},-s}\left\|\nabla_hw \right\|_{0,s-1}\left\|\eta \right\|_{\frac{1}{2},1-s}.
\end{align*}
Finally, for the last term, we use the fact that $s\leq 1$ which implies that $2s-1\leq 1$, and that $w$ is a free divergence vector field to infer that
\begin{align*}
L_9&=2^{-2(s-1)}\langle S_0^v \theta ,S_0^v w^3\rangle + \displaystyle\sum_{q\geq 0} 2^{2q(s-1)} \langle \Delta_q^v \theta , \Delta_q^v w^3 \rangle\\
& \lesssim \norm {S_0^v\theta}_{L^2}\norm {S_0^vw}_{L^2}+\displaystyle\sum_{q\geq 0}2^{-qs}\norm {\Delta_q^v \theta}_{L^2} 2^{q(s-1)} 2^{q(2s-1)} \norm {\Delta_q^v w^3}_{L^2}\\
&\lesssim \norm {S_0^v\theta}_{L^2}\norm {S_0^vw}_{L^2}+\displaystyle\sum_{q\geq 0}2^{-qs}\norm {\Delta_q^v \theta}_{L^2} 2^{q(s-1)} \norm {\Delta_q^v \nabla_h \cdot w^h}_{L^2}\\
&\lesssim \norm {\theta}_{0,-s}\big(\norm {\nabla_h w}_{0,s-1}+ \norm {w}_{0,s-1} \big).
\end{align*}
\end{proof}
\noindent In the case where $s= \frac{1}{2}$, the estimates are more delicate since the space $H^\frac{1}{2}(\mathbb{R})$ is not an algebra. In the present paper, and in this particular case, we will just take up again the reasoning due to M. Paicu in \cite{Paicu2}, and use his estimates to treat all the equations in the same energy space of $H^{0,-\frac{1}{2}}$. \\ More precisely one may prove the following proposition:
\begin{prop} \label{prop2}
Let $u,v,\rho$ and $\eta$ be in space $L^{\infty}_T(H^{0,\frac{1}{2}})$ with $\nabla_h u, \nabla_h v, \nabla_h \rho,\nabla_h \eta$ in $L^{2}_T(H^{0,\frac{1}{2}})$ and $u,v$ two divergence free vector fields. Let $w,\theta$ be in $L^{\infty}_T(H^{0,\frac{1}{2}})$ with $\nabla_hw$ and $\nabla_h\theta$ in $L^{\infty}_T(H^{0,\frac{1}{2}})$ solution to the following equations
\begin{equation}\label{equa-unicité}
 \left\{\begin{array}{l}
\partial_t w+u\cdot\nabla w -\Delta_hw+\nabla \varpi=\theta e_3-w\cdot\nabla v,\\
\partial_t \theta+u\cdot\nabla\theta -\Delta_h\theta =-w\cdot\nabla \eta,\\
\div w=0.\\
\end{array}\right.,
\end{equation}
Let $\chi(t)\bydef \norm {w(t)}_{0,-\frac{1}{2}}^2+ \norm {\theta(t)}_{0,-\frac{1}{2}}^2$. If $\chi(t) \leq e^{-2}$
 , for all $0<t<T$, then we have:
\begin{equation}\label{cases=1/2}
\frac{d}{dt}\chi(t)\leq C f(t)\chi(t) \big(1- ln \chi(t) \big)ln\big(1- ln \chi(t)\big),
\end{equation}
where $f$ is a locally integrable function depending on the norms of $u,v,w,\rho,\eta,\theta$ in $H^{0,\frac{1}{2}}\cap H^{1,\frac{1}{2}}$.  
\end{prop}
\begin{proof}
As mentioned before, the strategy is similar to the case $s>\frac{1}{2}$   but  it needs more attention because of the lack of the embedding $H^\frac{1}{2}(\mathbb{R})\hookrightarrow L^\infty(\mathbb{R})$.  
The proof we propose here is a direct application of the estimates proved in \cite{Paicu2}. In particular, note  that the terms $u\cdot \nabla w$ and $u\cdot \nabla\theta$ (resp. $w \cdot \nabla v$ and $w\cdot \nabla \eta)$ can be treated along the same way since we assume here that $u$ and $\rho$ (resp. $v$ and $\eta$) have the same regularity. Some details are given below.\\
Let us recall the following estimate from the proof of Lemma 4.2 from \cite{Paicu2}:
\begin{align}\label{unicité-estimate1}
\sum_{q\geq -1} 2^{-q}  \bigg(\langle \Delta_q^v (u \cdot \nabla w), \Delta_q^v w  \rangle &+\langle \Delta_q^v (w \cdot \nabla v), \Delta_q^v w  \rangle  \bigg) \leq  \frac{1}{50}  \norm {\nabla_hw}_{H^{0,-\frac{1}{2}}}^2  \nonumber \\
&+Cf_1(t)\norm w_{H^{0,-\frac{1}{2}}}^2 (1-  \norm w_{H^{0,-\frac{1}{2}}}^2) ln (1- ln \norm w_{H^{0,-\frac{1}{2}}}^2)),
\end{align}
where 
$$f_1   \bydef \big(1+\norm u_{1,\frac{1}{2}}^2 + \norm v_{1,\frac{1}{2}}^2 + \norm w_{1,\frac{1}{2}}^2 \big) \times \big(1+\norm {\nabla_h u}_{1,\frac{1}{2}}^2 + \norm {\nabla_h v}_{1,\frac{1}{2}}^2 + \norm {\nabla_h w}_{1,\frac{1}{2}}^2\big).$$
Similar arguments can be used to establish the following estimate for the the second equation in \eqref{equa-unicité}
\begin{align}\label{unicité-estimate2}
\sum_{q\geq -1} 2^{-q}  \bigg(\langle \Delta_q^v (u \cdot \nabla \theta), \Delta_q^v \theta  \rangle &+\langle \Delta_q^v (\theta \cdot \nabla \eta), \Delta_q^v \theta  \rangle  \bigg) \leq  \frac{1}{50}  \norm {\nabla_h\theta}_{H^{0,-\frac{1}{2}}}^2  \nonumber \\
&+Cf_2(t)\norm \theta_{H^{0,-\frac{1}{2}}}^2 (1-  \norm \theta_{H^{0,-\frac{1}{2}}}^2) ln (1- ln \norm \theta_{H^{0,-\frac{1}{2}}}^2)),
\end{align}
where 
$$f_2  \bydef \big(1+\norm u_{1,\frac{1}{2}}^2 + \norm \eta_{1,\frac{1}{2}}^2 + \norm \theta_{1,\frac{1}{2}}^2 \big) \times \big(1+\norm {\nabla_h u}_{1,\frac{1}{2}}^2 + \norm {\nabla_h \eta}_{1,\frac{1}{2}}^2 + \norm {\nabla_h \theta}_{1,\frac{1}{2}}^2\big).$$
Both estimates \eqref{unicité-estimate1} and \eqref{unicité-estimate2} hold under the assumption $\chi(t)= \norm {w(t)}_{0,-\frac{1}{2}}^2+ \norm {\theta(t)}_{0,-\frac{1}{2}}^2 \leq e^{-2}.$
Finally, the estimate of $\theta e_3$ is the following
\begin{equation}\label{unicité-force term}
\sum_{q\geq -1} 2^{-q}   \langle \Delta_q^v  \theta e_3, \Delta_q^v w  \rangle \leq \norm w_{H^{0,-\frac{1}{2}}} \norm \theta_{H^{0,-\frac{1}{2}}} \leq \norm w_{H^{0,-\frac{1}{2}}}^2+ \norm \theta_{H^{0,-\frac{1}{2}}}^2. 
\end{equation} 
On the other hand, for $0<x\ll1$, the function $x \longmapsto x(1-ln(x)) ln(1-ln(x))$ is non-decreasing. 
It is then easy to deduce \eqref{cases=1/2} from  \eqref{unicité-estimate1}, \eqref{unicité-estimate2} and \eqref{unicité-force term}  by setting 
$$f(t)\bydef f_1(t)+ f_2(t) + 1.$$
\end{proof}
\section{Proof of the main Theorems }

\subsection{Proof of Theorem \ref{th1}} 
\begin{proof}[]
Let $(u,\rho,P),(v,\eta,\Pi)$ be two solutions for system \eqref{NSB_h}, and $w\bydef u-v,\theta\bydef\rho-\eta,\varpi\bydef P-\Pi$ denote the difference functions. Then $(w,\theta,\varpi)$ satisfies 
\begin{equation}\label{mathcal{Q}}
\left\{\begin{array}{l}
\big(\partial_t+u\cdot\nabla  \big)w -\Delta_hw+\nabla \varpi=\theta e_3-w\cdot\nabla v,\\
\big(\partial_t+u\cdot\nabla  \big)\theta -\Delta_h\theta =-w\cdot\nabla \eta,\\
\div w=0,\\
w_{|t=0}=\theta_{|t=0}=0. \tag{$\mathcal{Q}$}
\end{array}\right.
\end{equation} 
\noindent$\bullet$~\textit{\textbf{The case: $s\neq \frac{1}{2}$}}\\
In the sequel, the constant $C$ denotes a universal constant which is allowed to differ from line to line.\\
Recall first that, by interpolation, we have:
\begin{equation*}
\norm {w}_{\frac{1}{2},s-1}\lesssim \norm {w}_{0,s-1}^{\frac{1}{2}} \norm {\nabla_h w}_{0,s-1}^{\frac{1}{2}},
\end{equation*}
\begin{equation*}
\norm {\theta}_{\frac{1}{2},-s}\lesssim \norm {\theta}_{0,-s}^{\frac{1}{2}} \norm {\nabla_h \theta}_{0,-s}^{\frac{1}{2}}.
\end{equation*}
By using the above interpolation result with the estimates from Proposition \eqref{proposition1}, we infer that
\begin{equation}\label{ES1}
L_1+ L_5 \lesssim \norm u _{\frac{1}{2},s} \norm {\nabla_h w}_{0,s-1}^\frac{3}{2} \norm w_{0,s-1}^\frac{1}{2},
\end{equation}
\begin{align}\label{ES2}
\sum_{i\in\{2,3,6,7\}}L_i \lesssim \big( \norm{\nabla_h u }_{0,s } +\norm{\nabla_h v }_{0, s} +\norm{\nabla_h \eta }_{0, 1-s} \big)&\big( \norm  w _{0,s-1} + \norm \theta_{0,-s} \big)\nonumber\\& \times\big( \norm {\nabla_h w} _{0,s-1} + \norm {\nabla_h \theta}_{0,-s} \big), 
\end{align}
\begin{align}\label{ES3}
L_4 + L_8 \lesssim \big( \norm{ v }_{\frac{1}{2}, s}& +\norm{ \eta }_{\frac{1}{2},1-s}  \big) \bigg[ \big( \norm w_{0,s-1} + \norm \theta_{0,-s} \big)^\frac{3}{2} \big( \norm {\nabla_h w}_{0,s-1}+ \norm {\nabla_h w}_{0,s-1}\big)^\frac{1}{2}\nonumber\\
&  +  \big( \norm w_{0,s-1} + \norm \theta_{0,-s} \big)^\frac{1}{2} \big( \norm {\nabla_h w}_{0,s-1}+ \norm {\nabla_h w}_{0,s-1}\big)^\frac{3}{2} \bigg].
\end{align} 
Finally, we recall the estimate of $L_9$
\begin{equation}\label{ES4}
L_9 \lesssim \norm {\theta}_{0,-s} \norm {\nabla_h w}_{0,s-1}+ \norm {\theta}_{0,-s} \norm {w}_{0,s-1}.
\end{equation}
An easy consequence of the Young inequality tells that, for any non negative real numbers $\alpha, \beta,A,B,D$ with $\alpha+ \beta=2$  , we have

\begin{equation*}
DA^{\alpha}B^{\beta} \leq \frac{1}{100} B^2 + C A^2 D^{2/\alpha}.
\end{equation*}
 By a suitable choice of $\alpha$, $\beta$ in each one of the estimates \eqref{ES1},\eqref{ES2}, \eqref{ES3} and \eqref{ES4}, we infer that
\begin{equation}
\sum_{i=1}^9 L_i \leq \frac{1}{2} \big(\norm {\nabla_h w}_{0,s-1}^2 +  \norm {\nabla_h \theta}_{0,-s}^2\big)+ Cf(t)\big(\norm {w}_{0,s-1}^2 +  \norm {\theta}_{0,-s}^2\big),\label{eq27}
\end{equation}
where $f$ is a function, locally integrable in time \footnote{Remark that the assumption on $u,v$ being $L^\infty_T(H^{0,s)} \cap L^2_T(H^{1,s})$ is enough to ensure, by interpolation argument, that $u$ and $v$ belong to $L^4_T(H^{\frac{1}{2} ,s}). $},   given by:
\begin{equation*}
f = \norm {u}_{\frac{1}{2},s}^4+\norm {\nabla_h u}_{0,s}^2+\norm {\nabla_h v}_{0,s}^2 +\norm {\nabla_h \eta}_{0,1-s}^2 +\norm {v}_{\frac{1}{2},s}^{\frac{4}{3}}+\norm {\eta}_{\frac{1}{2},1-s}^\frac{4}{3} +\norm {v}_{\frac{1}{2},s}^4+\norm {\eta}_{\frac{1}{2},s}^4+1 .
\end{equation*}
On the other hand, by applying the operator $\Delta_q^v$ to \eqref{mathcal{Q}}, and by summing with respect to 
$q\in\mathbb{N}\cup \{-1\}$, then \eqref{eq27} leads to
$$
\norm {w(t)}_{0,s-1}^2+ \norm {\theta(t)}_{0,-s}^2 \leq \displaystyle C\int_{0}^t f(\tau)\big(\norm {w(\tau)}_{0,s-1}^2+ \norm {\theta(\tau)}_{0,-s}^2\big)  d\tau.
$$
Finally, we can conclude by using Gronwall's Lemma.\\

\noindent$\bullet$~\textit{\textbf{The case: $s= \frac{1}{2}$}}\\
\noindent The uniqueness in this case can be deduced by applying the Osgood's Lemma to the estimate \eqref{cases=1/2} given in Proposition \ref{prop2}.
\end{proof}
\subsection{Proof of Theorem \ref{th2}}
\begin{proof}[]\\
\noindent The uniqueness result in Theorem \ref{th2} is a direct consequence of Theorem \ref{th1} when we take $s=1$. 
Thus, we only have to prove the existence of a global solution $(u,\rho)$ for ($NSB_h$). The arguments given hereafter, based on the Friedrich's method, are very classical. (see for instance \cite{Chemin1,Paicu1,Paicu2,Danchin1} for more details)\\ 
  
\noindent For $n\in\mathbb N$, we consider the following approximate system
\begin{equation}\label{B_n}
\left\{\begin{array}{l}
\partial_t u_n+ \mathbb{E}_n (u_n\cdot\nabla u_n) -\Delta_hu_n +\nabla P_n =\rho_n  e_3,\\
\partial_t \rho_n + \mathbb{E}_n (u_n \cdot\nabla \rho_n)  -\Delta_h\rho_n  =0,\\
\div u_n=0,\\
P_n=  \mathbb{E}_n \sum_{k,j} (-\Delta)^{-1} \partial_j\partial_k(u_n^j u_n^k),\\
(u_n,\rho_n)_{|t=0}=(\mathbb{E}_nu_0,\mathbb{E}_n\rho_0),\\
\end{array}\right. \tag{$B_n$}
\end{equation}
 
\noindent where $\mathbb{E}_n$ denotes the cut-off operator defined on $L^2(\mathbb{R}^3)$ by 
$\mathbb{E}_n u \bydef \mathcal{F}^{-1} (\mathds{1}_{B(0,n)} \hat{u})$. \\ 
It is then easy to see by using a fixed point argument that there exists some $T_n>0$ for which  
\eqref{B_n} admits a unique solution $(u_n,\rho_n)\in\mathcal{C}^{\infty}([0,T_n[, \mathcal{L}^{2,\sigma}_n)$, where 
$$\mathcal{L}^{2,\sigma}_n\bydef L^{2,\sigma}_n\times L^{2}_n(\mathbb{R}^3),$$
$$
 L^{2,\sigma}_n  \bydef \big\{ v\in \big(L^2(\mathbb{R}^3)\big)^3  : \div v = 0 \; and \; Supp\; (\hat{v}) \subset B(0,n) \big\},
$$
$$L^{2}_n(\mathbb{R}^3)\bydef \big\{ \rho\in L^2(\mathbb{R}^3) : \text{Supp }\; (\hat{\rho}) \subset B(0,n) \big\}.$$
Moreover, because $u_n$ and $\rho_n$ are regular, we can multiply the first equation in \eqref{B_n} by $u_n$ and the second one by $\rho_n$. Then, for $t\in [0,T_n[$, after integrating the corresponding terms over $[0,t[\times\mathbb R^3$, we obtain the classical uniform $L^2$-energy bounds 
\begin{align}
&\norm {u_n(t)}_{L^2}^2 +  2\displaystyle\int_0^t \norm {\nabla_h u_n(\tau)}_{L^2}^2d\tau\leq 2(\norm {u_0}_{L^2}^2+t^2  \norm {\rho_0}_{L^2}^2),\label{uL2energy} \\ 
&\norm {\rho_n(t)}_{L^2}^2 +  2\displaystyle\int_0^t \norm {\nabla_h \rho_n(\tau)}_{L^2}^2d\tau\leq \norm {\rho_0}_{L^2}^2.\label{rhoL2energy}
\end{align}
Hence, $(u_n,\rho_n)$ is a global solution, that is for any $T>0,\;  (u_n,\rho_n)\in\mathcal{C}^{\infty}([0,T[, \mathcal{L}^{2,\sigma}_n)$ and it satisfies \eqref{B_n} on $[0,T[\times\mathbb R^3$. Moreover, we may extract a 
sub-sequence, still denoted $(u_n,\rho_n)$, such that
$$(u_n,\rho_n) \overset{*}{\rightharpoonup} (u,\rho) \quad \text{in } L^\infty_TL^2\cap L^2_TH^{1,0}.$$
However, in order to pass to the limit in the non-linear terms, we will need some strong convergence property. To this end we can use the Proposition 3.2 established in \cite{Miao}. We obtain
$$\|u_n(t)\|_{H^1}^2 +  \displaystyle\int_0^t \norm {\nabla_h u_n(\tau)}_{H^1}^2d\tau\leq C_0e^{C_0t}.$$
It follows that $u_n$ is uniformly bounded in $L^\infty_TH^1\cap 
L^2_T(H^{1,1}\cap H^{2,0})$. \\  
Assume temporarily that
\begin{equation}\label{injection}
\big( L^{\infty}_TL^2 \cap L^2_T H^{1,0}\big)  \cap \big(L^{\infty}_TH^{0,1} \cap L^2_T H^{1,1} \big) \hookrightarrow L^4_TL^{2}_v(L^4_h) \cap L^4_T L^{\infty}_v(L^4_h) .
\end{equation}
We infer that
 \begin{itemize}
 \item $(u_n)$ is bounded in $ L^4_TL^{2}_v(L^4_h) \cap L^4_T L^{\infty}_v(L^4_h) .$
 \item $(\rho_n)$ is bounded in $L^{\infty}_TL^2 \cap L^2_T H^{1,0} \hookrightarrow L^4_TL^{2}_v(L^4_h)$.
 \end{itemize}
Hence, $\nabla P_n$, $\div(u_n\rho_n)$ and $\div(u_n \otimes u_n)$ are bounded in $L^2_TH^{-1}$, and this gives a bound for $\partial_t u_n$ and $\partial_t \rho_n$  in $L^2_T H^{-1}$. We can then use Aubin-Lions Theorem in order to extract a new sub-sequence (still denoted $(u_n,\rho_n)$) that strongly converges to $(u,\rho)$ in $L^2_TH^{\frac{1}{2}}_{loc} \times L^2_T(H^{-\frac{1}{2}}_{loc})$. Now, we can pass to the limit $n\to\infty$ in all the terms in \eqref{B_n}, and we show that 
$(u,\rho)$ is a solution to $\eqref{NSB_h}$. \\ 

\noindent In order to be more convenient, we briefly outline some details about how to pass to the limit in the non linear term in the $\rho$-equation: Let $\psi,\zeta$ be smooth functions in $C_c^\infty(\mathbb{R}^+ \times\mathbb{R}^3)$, supported respectively in $K\Subset   \mathbb{R}^+ \times\mathbb{R}^3$, and $\tilde{K}\Subset   \mathbb{R}^+ \times\mathbb{R}^3$, with $K $ being strictly included in $\tilde{K}$ and $\zeta|_{K}=1$, we claim 
\begin{equation}\label{compactness-limite}
\lim_{n \rightarrow \infty} \int_{(0,T)\times \mathbb{R}^3} u_n \rho_n \cdot \nabla \psi = \int_{(0,T)\times \mathbb{R}^3} u  \rho  \cdot \nabla \psi.
\end{equation}
 According to the properties of $\psi$ and $\zeta$ we have
 $$\int_{(0,T)\times \mathbb{R}^3} u_n \rho_n \cdot \nabla \psi  = \int_{\tilde{K} } (\zeta u_n) (\rho_n \cdot \nabla \psi). $$
 We recall that, from the uniform bounds proved above, we have for $s>\frac{1}{2}$
 \begin{itemize}
 \item $(u_n)$ is bounded in $L^2_T(H^s)$ and $(\partial_tu_n)$ is bounded in $L^2_T(H^{-1})$.
  \item $(\rho_n)$ is bounded in $L^2_T(L^2)$ and $(\partial_t\rho_n)$ is bounded in $L^2_T(H^{-1})$.
 \end{itemize}
 hence, by simple product laws we obtain the same bounds for $(\zeta u_n )$ and $(\rho _n \cdot \nabla \psi)$ as $ (u_n)$ and $(\rho_n)$ respectively. Therefore, by combining the Aubin-Lions Theorem, Theorem 2.94 from \cite{Chemin1} \footnote{Theorem 2.94 from \cite{Chemin1} gives in fact a more general result in the Besov context, what we would apply here is a particular case which says that the multiplication operator by a smooth function is a compact operator from $H^s$ into $H^{s-\varepsilon}$, for any $s\in \mathbb{R}$ and $\varepsilon>0$.} and the weak compactness property of the Hilbert spaces, we infer that 
 \begin{itemize}
 \item $(\zeta u_n )$ converges strongly to $(\zeta u  )$ in $L^2_T(H^\frac{1}{2})$,
 \item $(\rho_n   \cdot \nabla \psi )$ converges strongly to $(\rho   \cdot \nabla \psi )$ in $L^2_T(H^{-\frac{1}{2}})$.
 \end{itemize}
 This should be enough to justify assertion \eqref{compactness-limite}.\\
 
\noindent To complete our proof, we shall justify \eqref{injection}. To this end, we will prove a more general inequality: Let $f\in \mathcal{S}'(\mathbb{R}^3)$ and $s>\frac{1}{2}$ (note that in inequality \eqref{injection} $s$ is equal to $1$). By using Lemma \ref{ber}, we obtain for some non-negative number $N$ to be fixed later
\begin{align*}
 \norm f_{L^4_h(L^{\infty}_v)}&\leq \sum_{k\geq -1} \sum_{j\geq -1}\norm {\Delta_k^h \Delta_j^v f}_{L^4_h(L^{\infty}_v)}\\
 &\leq  \sum_{N\geq k\geq -1} \sum_{j\geq -1} 2^{\frac{k}{2}} 2^{j(\frac{1}{2}-s)}2^{js}\norm {\Delta_k^h \Delta_j^v f}_{L^2} + \sum_{k \geq N+1} \sum_{j\geq -1} 2^{\frac{-k}{2}}2^{j(\frac{1}{2}-s)}2^{js}\norm { \nabla_h \Delta_k^h \Delta_j^v f}_{L^2}.
\end{align*}
Hence, because $s>\frac{1}{2}$ and $j\geq -1$, the Cauchy-Schwarz inequality gives
\begin{align*}
\norm f_{L^4_h(L^{\infty}_v)}&\lesssim 2^{\frac{N}{2}} \norm f_{L^2_h(H^s_v)} + 2^{-\frac{N}{2}} \norm {\nabla_hf}_{L^2_h(H^s_v)} .
\end{align*}
Therefore, by choosing $N$ such that $2^N= \frac{\norm {\nabla_hf}_{L^2_h(H^s_v)}}{\norm f_{L^2_h(H^s_v)}}$, we obtain
$$
\norm f_{L^4_h(L^{\infty}_v)} \lesssim \norm f_{L^2_h(H^s_v)} ^{\frac{1}{2}} \norm {\nabla_hf}_{L^2_h(H^s_v)}^{\frac{1}{2}}.
$$
On the other hand, the Minkowski inequality ensures that $L^4_h(L^{\infty}_v) \hookrightarrow L^{\infty}_v(L^4_h)
$. This is sufficient to conclude the proof of \eqref{injection}. 
\end{proof}
\appendix
\section{Appendix}
In this additional section we prove a result of well posedness for \eqref{NSB_h} under some smallness conditions on: $T$, the norm of $u_0$ in $H^{0,s}(\mathbb{R}^3)$ and the norm of $\rho_0$ in $L^2(\mathbb{R}^3)$. This result may not be optimal in this direction, but it gives the existence of solutions in a some new\footnote{We have already treated in Theorem \ref{th2} some particular situations with axisymmetric initial data.} situations where the proof of the uniqueness part is reduced to an application of Theorem \ref{th1}.
\begin{theo}{(\textbf{Local well-posedness)}} \label{th3} \\
Let $s\in]\frac{1}{2},1]$ , $\delta\in [0,s]$ and $(u_0,\rho_0)\in H^{0,s} \times H^{0,\delta}$. We have:
\begin{itemize}
\item There exists $C_s> 0$ such that if $$\norm {u_0}_{0,s}^2+ T\norm{\rho_0}_{L^2}\bigg(\norm{u_0}_{L^2}+\norm{\rho_0}_{L^2}\big(1+\frac{T}{2}\big) \bigg)< C_s^2.$$ then \eqref{NSB_h} has at least one solution $(u,\rho)$ in $\mathcal{X}^{s,\delta}(T)$, where:
$$\mathcal{X}^{s,\delta}(T)\bydef L^{\infty}_TH^{0,s}\cap L^{2}_TH^{1,s}\times L^{\infty}_TH^{0,\delta}\cap L^{2}_TH^{1,\delta}.$$
\item The solution is unique if $\delta \geq 1-s$.
\end{itemize} 
\end{theo}
\begin{rmq}
Notice that, according to Theorem \ref{th3}, the construction of solutions to \eqref{NSB_h} (at least locally in time) does not require high regularity for the initial density $\rho_0$. This is in fact possible due to the high vertical regularity of the velocity which allows us to rigorously justify the a priori estimates (as in the Friedrich's method explained in the proof of Theorem \ref{th2}). 
\end{rmq}
\begin{rmq}
The constant $C_s$ in Theorem \ref{th3} tends in fact to zero when $s$ tends to $\frac{1}{2}$. In this case, our a priori estimates fall, and hence the existence of solutions in the case $s=\frac{1}{2}$ remains open, even for the classical Navier-Stokes equations.
\end{rmq}
\begin{rmq}
In the case where $s=\frac{1}{2}$, one may prove a local well-posedness result if the initial data $(u_0,\rho_0)$ is in $\big(L^2_h(B^{\frac{1}{2}}_{2,1})_v\big)^2$. See for instance \cite{Paicu1} where this is established for $(NS_h)$.
\end{rmq}~~\\
\noindent The proof of the existence part in Theorem \ref{th3} in based on the following Lemma:
\begin{lemma} \label{lastlemma}
Let $s\in ]\frac{1}{2},1]$, $\delta\in [0,s]$, then for all regulars vector fields $a,b$ with $\div a=0$, we have
\begin{align}
\big|\left\langle a\cdot \nabla b , b \right\rangle_{0,\delta}\big| &\lesssim \norm b_{\frac{1}{2},\delta}\big(\norm a_{1,s} \norm b _{\frac{1}{2},\delta} + \norm a_{\frac{1}{2},s} \norm b _{1,\delta}\big). \label{eq33}\\ 
\big|\left\langle \rho,a^3 \right\rangle_{0,s}\big|&\leq \frac{1}{4} \norm {\rho}_{L^2} \norm {a}_{L^2} + \norm {\rho}_{L^2}\norm {\nabla_h a}_{0,s}. \label{eq34}
\end{align}
\end{lemma}
\begin{proof}
In order to prove \eqref{eq33}, we follow the same approach as in \cite{Chemin}, so we write
\begin{equation}
\left\langle a\cdot \nabla b ,  b \right\rangle_{0,\delta} = \left\langle  a^h\cdot \nabla_h b ,  b \right\rangle_{0,\delta} +\left\langle  a^3 \partial_3 b , b \right\rangle _{0,\delta}. \label{eq35}
\end{equation}
We remark next that Lemma \ref{productrule1} gives $H^{\frac{1}{2},s} \times H^{0,\delta} \hookrightarrow H^{-\frac{1}{2},\delta}$, which implies 
\begin{align*}
\big|\left\langle  a^h\cdot \nabla_h b ,  b \right\rangle_{0,\delta}\big| &\lesssim \norm b_{\frac{1}{2},\delta}\norm {a^h\cdot \nabla_h b}_{-\frac{1}{2},\delta}\\
&\lesssim \norm b_{\frac{1}{2},\delta}\norm a_{\frac{1}{2},s} \norm b_{1,\delta} .
\end{align*}
It remains now only to estimate the second term in the right hand side of \eqref{eq35}. Indeed, Bony's decomposition tells that
\begin{equation}\label{000}
\Delta_q^v(a^3\partial_3b)= \Delta_q^v\bigg(\sum_{k\geq q-N_0}S_{k+2}(\partial_3b) \Delta_k^v a^3 + \sum_{|q-k|\leq N_0}S_{k-1}a^3 \Delta_k^vb \bigg).
\end{equation}
Note then firstly that 
\begin{align*}
\left\| S_{k+2}(\partial_3 b)\right\|_{L^4_hL^2_v}&\leq  \displaystyle\sum_{m\leq k+1} 2^{m(1-\delta)}2^{\delta m}\left\|\Delta_m^vb\right\|_{L^4_hL^2_v} 
\leq \bigg(\displaystyle\sum_{m\leq k+1} 2^{2m(1-\delta)}\bigg)^{\frac{1}{2}}\left\|b \right\|_{\frac{1}{2},\delta} \\
&\lesssim 2^{k(1-\delta)}\left\|b \right\|_{\frac{1}{2},\delta}.
\end{align*}
Moreover, by using the fact that $\partial_3 a^3 = - \nabla_h \cdot u^h$ together with Lemma \ref{ber}, we obtain
\begin{align*}
\norm {\Delta_k^v a^3}_{L^2}&\leq C 2^{-k}\norm {\Delta_k^v  \nabla_h \cdot u^h}_{L^2}.
\end{align*}
This implies
\begin{align*}
\bigg|\left\langle \Delta_q^v \sum_{k\geq q-N_0}S_{k+2}(\partial_3b) \Delta_k^v a^3 , \Delta_q^v b \right\rangle \bigg|&\lesssim \norm {\Delta_q^v b}_{L^4_hL^2_v} 2^{q/2}\sum_{k\geq q-N_0}\norm {S_{k+2}(\partial_3b)}_{L^4_hL^2_v} \norm {\Delta_k^v a^3}_{L^2}\\
& \lesssim \norm {\Delta_q^v b}_{L^4_hL^2_v} 2^{q/2}\sum_{k\geq q-N_0}2^{-k\delta}\left\|b \right\|_{\frac{1}{2},\delta}\norm {\Delta_k ^v \nabla_h u}_{L^2}\\
&\lesssim \left\|b \right\|_{\frac{1}{2},\delta} \norm {u}_{1,s} 2^{-q\delta}c_q \norm { b}_{\frac{1}{2},\delta} 2^{q/2}\sum_{k\geq q-N_0}2^{-k(\delta+s)} c_k\\
&\lesssim \left\|b \right\|_{\frac{1}{2},\delta}^2 \norm {u}_{1,s} 2^{-2q\delta}c_q 2^{q(\frac{1}{2}-s)}.
\end{align*}
Hence, thanks to the assumption $s>\frac{1}{2}$, we infer that
\begin{equation*}
\bigg|\left\langle \sum_{k\geq -1}S_{k+2}(\partial_3b) \Delta_k^v a^3 , b \right\rangle_{0,\delta}\bigg|\lesssim \left\|b \right\|_{\frac{1}{2},\delta}^2 \norm {u}_{1,s}.
\end{equation*}
In order to estimate the second term in the right hand side of \eqref{000}, we use the decomposition proposed in \cite{Chemin}
\begin{align*}%\label{0000}
&\langle \Delta_q^v\sum_{|q-k|\leq N_0}S_{k-1}a^3 \Delta_k^vb, \Delta_q^v b\rangle =\langle S_q(a^3)\partial_3\Delta_q^v b, \Delta_q^v b\rangle + \langle\sum_{|q-k|\leq N_0} \big[\Delta_q^v, S_{k-1}a^3 \big] \partial_3\Delta_k^v b, \Delta_q^v b\rangle \\ 
&\phantom{tototototo} + \langle\sum_{|q-k|\leq N_0}\big(S_q -S_{k-1}\big)a^3 \Delta_k^v \partial_3b, \Delta_q^v b\rangle 
\bydef J_1^q+J_2^q+J_3^q.
\end{align*} 
By using an integration by parts, we obtain
$$
J_1^q = \left\langle S_q(a^3)\partial_3\Delta_q^v b , \Delta_q^v b \right\rangle = \frac{1}{2} \left\langle S_q(\partial _3a^3)\Delta_q^v b , \Delta_q^v b \right\rangle,
$$
which shows that this term can be estimated in the same way as $\langle \nabla_h a\cdot  b ,  b\rangle$ due to the divergence free condition. We get
$$
\big|\sum_{q\geq -1} 2^{2q\delta} \left\langle S_q(a^3)\partial_3\Delta_q^v b , \Delta_q^v b \right\rangle \big|\lesssim \norm a_{1,s} \norm b_{\frac{1}{2},\delta}^2.
$$
Next from Lemma \ref{commutator.lemma}, it easily follows
\begin{align*}
\big|\sum_{q\geq -1} 2^{2q\delta} J_2^q \big|= \bigg|\sum_{q\geq -1} 2^{2q\delta} \left\langle \sum_{|q-k|\leq N_0} \big[\Delta_q^v, S_{k-1}a^3 \big] \partial_3\Delta_k^v b , \Delta_q^v b \right\rangle\bigg| &\lesssim \norm a_{1,s}\sum_{q\geq -1} 2^{2q\delta} \norm {\Delta_q^v b}_{\frac{1}{2},0}\\
&\lesssim \norm a_{1,s} \norm b _{\frac{1}{2},\delta}^2.
\end{align*}
Finally, to estimate $J_3^q$, we use the fact that the support of $S_q^v-S_{k-1}^v$ is far from zero in Fourier side.  This gives in particular
$$\norm {(S_q^v-S^v_{k-1}) a^3}_{L^2_h(L^\infty_v)} \lesssim 2^{-k} \norm {(S_q^v-S^v_{k-1})\partial_3 a^3}_{L^2_h(L^\infty_v)}. $$
Thus, we obtain
\begin{align*}
J_3^q \lesssim \sum_{|q-k|\leq N_0} 2^{q-k} \norm {(S_q^v-S^v_{k-1})\partial_3 a^3}_{L^2_h(L^\infty_v)} \norm{\Delta_k^v b}_{L^4_h(L^2_v)} \norm{\Delta_q^v b}_{L^4_h(L^2_v)} .
\end{align*}
By using the fact that $\div a =0 $ and $s>\frac{1}{2}$, we end up with
$$\big|\sum_{q\geq -1} 2^{2q\delta} J_3^q \big|\lesssim \norm a_{1,s} \norm b _{\frac{1}{2} ,\delta}^2.$$
 This ends the proof of \eqref{eq33}. \\
In order to prove \eqref{eq34}, we use again the fact that $\div a=0$ together with $s-1\leq 0$, to obtain
\begin{align*}
\big|\left\langle \rho,a^3 \right\rangle_{0,s}\big|&=\big|2^{-2s}\langle S_0^v \rho ,S_0^v a^3\rangle + \displaystyle\sum_{q\geq 0} 2^{2qs} \langle \Delta_q^v \rho , \Delta_q^v a^3 \rangle\big|\\
& \leq \frac{1}{4}\norm {S_0^v\rho}_{L^2}\norm {S^v_0a}_{L^2}+\displaystyle\sum_{q\geq 0}\norm {\Delta_q^v \rho}_{L^2} 2^{q(s-1)}  2^{qs}\norm {\Delta_q ^v\partial_3 a^3}_{L^2}\\
&\leq \frac{1}{4} \norm {S^v_0\rho}_{L^2}\norm {S^v_0a}_{L^2}+\displaystyle\sum_{q\geq 0}\norm {\Delta_q^v \rho}_{L^2} 2^{qs} \norm {\Delta_q^v \nabla_h a}_{L^2}\\
&\leq \frac{1}{4} \norm {\rho}_{L^2} \norm {a}_{L^2} + \norm {\rho}_{L^2}\norm {\nabla_h a}_{0,s}.
\end{align*}
This concludes the proof of Lemma \ref{lastlemma}
\end{proof}
\subsection*{Proof of Theorem \ref{th3}}
\begin{proof}[]\\
\noindent The uniqueness part of Theorem \ref{th3} is a direct consequence of Theorem \ref{th1}. Hence, it remains only to prove the existence part. This can be done in a similar way than explained in the proof of Theorem \ref{th2}, and in particular the construction of the approximate sequence $(u_n,\rho_n)$ does not add any difficulty. However, the uniform bounds must now be obtained in some new adequate norms. For simplicity, in the following we will drop the index of the approximate sequence. \\   

\noindent We apply $\Delta_q^v$ in both equations for $u$ and $\rho$ from \eqref{NSB_h}, then we multiply the first equation by $\Delta_q^v u$, the second one by $\Delta_q^v \rho$ and we sum over $q\geq -1$ to obtain
\begin{equation}\label{ueq}
\frac{d}{2dt}\norm {u(t)}_{0,s}^2+  \norm {\nabla_h u(t)}_{0,s}^2\leq \big|\left\langle u\cdot \nabla u, u \right\rangle_{0,s}\big|+ \big|\left\langle \rho, u^3 \right\rangle_{0,s}\big|.
\end{equation}
\begin{equation}\label{rhoeq}
\frac{d}{2dt}\norm {\rho(t)}_{0,\delta}^2+  \norm {\nabla_h \rho(t)}_{0,\delta}^2\leq\big| \left\langle u\cdot \nabla \rho, \rho \right\rangle_{0,\delta}\big|.
\end{equation}
We prove firstly an uniform bound for $u$ by using the $L^2$-energy estimate of $\rho$ and $u$. Indeed, by taking $a=b=u$ in Lemma \ref{lastlemma}, then \eqref{ueq} gives
\begin{equation}\label{exist}
\frac{d}{dt}\norm u_{0,s}^2+  2\norm {\nabla_h u}_{0,s}^2 \leq \norm {\rho}_{L^2} \norm u_{L^2} +  \norm {\nabla_h u}_{0,s}^2 + \norm {\rho}_{L^2}^2 + \widetilde{C}_s \norm u_{0,s}\norm {\nabla_h u}_{0,s}^2,
\end{equation}
where $\widetilde{C}_s$ is a non negative constant that depends on $s$.
Let us assume that
\begin{equation}\label{conditionT}
\norm {u_0}_{0,s}^2+ T\norm{\rho_0}_{L^2}\bigg(\norm{u_0}_{L^2}+\norm{\rho_0}_{L^2}\big(1+\frac{T}{2}\big) \bigg)< C_s^2 < \frac{1}{4\widetilde{C}_s^2},
\end{equation}
which ensures that $\norm {u_0}_{0,s}< C _s <\frac{1}{2 \widetilde{C}_s}$.\\
Let us now assume that there exists $T^{max}\in (0,T)$ satisfying
$$T^{max}\bydef\inf \{ t\in[0,T]: \norm {u(t)}_{0,s}=C_s\}.$$
It follows that, for all $ t\in [0,T^{max})$: $\norm {u(t)}_{0,s}<C_s < \frac{1}{2\widetilde{C}_s} $.\\
By using this last inequality in $\eqref{exist}$, and by integration on $[0,t)$, we obtain for all $t\in [0,T^{max})$
$$
\norm {u(t)}_{0,s}^2 + \frac{1}{2} \int_0^t\norm {\nabla_h u(\tau)}_{0,s}^2d\tau \leq \norm {u_0}_{0,s}^2 + \int_0^t \big( \norm {\rho(\tau)}_{L^2}^2 + \norm {u(\tau)}_{L^2}\norm {\rho(\tau)}_{L^2}\big) d\tau .
$$
Hence, by using the $L^2$ energy estimate for $\rho$ and $u$ given by \eqref{uL2energy} and \eqref{rhoL2energy}, we infer that 
$$
\norm {u(t)}_{0,s}^2 \leq \norm {u_0}_{0,s}^2+ t\norm{\rho_0}_{L^2}\bigg(\norm{u_0}_{L^2}+\norm{\rho_0}_{L^2}\big(1+\frac{t}{2}\big) \bigg), \quad \forall t\in [0,T^{max}),$$
then by using \eqref{conditionT} and passing to the limit $t\longrightarrow T^{max}$, we obtain
$$
C_s^2\leq \norm {u_0}_{0,s}^2+ T\norm{\rho_0}_{L^2}\bigg(\norm{u_0}_{L^2}+\norm{\rho_0}_{L^2}\big(1+\frac{T}{2}\big) \bigg)< C_s^2,
$$
which contradicts the existence of $T^{max}$ and gives for all $t\in [0,T]$: $\norm {u(t)}_{0,s}< C_s$. Therefore we have proved an uniform bound of $u$ in 
$L^\infty_TH^{0,s}$. Plugging this bound into \eqref{exist} gives a bound for $u$ in $L^2_TH^{1,s}$. Note that by using an argument of interpolation, this also gives a bound of $u$ in $L^{4}_TH^{\frac{1}{2},s}$. \\ 

\noindent Next, we can estimate the right hand side of \eqref{rhoeq} by using Lemma \ref{lastlemma} with $a=u$ and $b=\rho$. We obtain after some calculations
\begin{align*}
\frac{d}{dt}\norm {\rho(t)}_{0,\delta}^2 &+2\norm {\nabla_h \rho(\tau)}_{0,\delta}^2 d\tau \leq \norm {\nabla_h\rho}_{0,\delta}^2 + C_{s} A(t)\norm {\rho(t)}_{0,\delta}^2,
\end{align*} 
where 
$$ 
 A(\cdot)= \norm {\nabla_h u(\cdot)}_{0,s}^2 + \norm {u(\cdot)}_{\frac{1}{2},s}^4 \in L^1\big((0,T)\big).
$$
Finally, by applying Gronwall's Lemma we obtain the adequate bound for $\rho$.
\end{proof}
\section*{Acknowledgments} 
\begin{itemize}
\item The authors are very grateful to the referee for his/her valuable and helpful remarks and comments.
\item This work has been done when the second author was a PhD student at the University of C\^ote d'Azur - Nice - France, under the supervision of Fabrice Planchon and Pierre Dreyfuss. In particular, the second author would like  to thank his supervisors for the accomplished work.
\end{itemize}


\newpage
\begin{thebibliography}{}

\bibitem{Adhikari}
D. Adhikari, C. Cao, J .Wu: \textit{Global regularity results for the 2D Boussinesq equations
with vertical dissipation}, Journal of Differential Equations \textbf{251}, 1637-1655, (2011).

\bibitem{Chemin1}
H. Bahouri ,J.-Y. Chemin, R. Danchin, \textit{Fourier Analysis and Nonlinear Partial Differential Equations}, Springer-Verlag, Berlin-Heidelberg-Newyark, 2011.

\bibitem{Chemin}
J.-Y. Chemin,B. Desardins, I. Gallagher and E. Grenier: \textit{Fluid with anisotropic viscousity}, M2AN Math. Model. Numer. A,al., 34, pp. 315-335, 2000.

\bibitem{Chemin2}
J.-Y. Chemin, P. Zhang: \textit{On the critical one component regularity for 3-D Navier-Stokes system}, Z. Arch Rational Mech Anal (2017) 224: 871.

\bibitem{Danchin1}
R. Danchin, M. Paicu \textit{Les th\'eor\`emes de Leray et de Fujita-Kato pour le syst\`eme de Boussinesq partiellement visqueux}, 
Bulletin de la Soci\'et\'e Math\'ematique de France,  136 (2008), no. 2, pp. 261-309. doi : 10.24033/bsmf.2557.

\bibitem{HHZ}
A. Hanachi, H. Houamed, M. Zerguine: \textit{On the global well-posedness of the axisymmetric viscous Boussinesq system in critical Lebesgue spaces}, 
Discrete and Contin. Dyn. Syst. 40(2020), 6473-6506.

\bibitem{Hmidi1}
T. Hmidi, H. Abidi, S. Keraani: \textit{On the global regularity of axisymmetric Navier-Stokes-Boussinesq system}, 
Discrete and Contin. Dyn. Syst. 29 (2011), no 3, 737-756.

\bibitem{Hmidi2}
T. Hmidi, S. Keraani, F. Rousset: \textit{Global well-posedness for a Boussinesq-Navier-Stokes system with critical dissipation}, 
Communications in Partial Differential Equations 36 (2011), no 3, 420 – 445.

\bibitem{Haroune}
H. Houamed: \textit{About some possible blow-up conditions for the 3-D Navier-Stokes equations,} Accepted for publication in Journal of Differential Equations. arXiv: 1904.12485.

\bibitem{HZ}
H. Houamed, M. Zerguine: \textit{On the global solvability of the axisymmetric Boussinesq system with critical regularity,}  Nonlinear Analysis 200 (2020), Article 112003.


\bibitem{Ifti1}
D. Iftimie: \textit{The 3D Navier-Stokes equations seen as a perturbation of the 2D Navier-Stokes equations}, Bulletin de la S. M. F., tome 127, $n^o4$ (1999), p. 473-517.

\bibitem{Ifti2}
D. Iftimie: \textit{The resolution of the Navier-Stokes equations in anisotropic spaces}, Revista Mathematica Iberoamericana, vol. 15,  $n^o1$ (1999)

\bibitem{Ifti3}
D. Iftimie: \textit{A uniqueness result for the Navier-Stokes equations with vanishing vertical viscosity}, 
SIAM J. Math. Analysis, Vol. 33, No. 6 : pp. 1483-1493, 2002.

\bibitem{Miao}
C. Miao, X. Zheng: \textit{On the global well-posedness for the Boussinesq system with horizontal dissipation}, (2013) Communications in Mathematical Physics, 321(1), 33-67.

\bibitem{Paicu1}
M. Paicu, \textit{Equation anisotrope de Navier-Stokes dans des espaces critiques}, Revista Mathematica Iberoamericana \textbf{21} (2005), $n^o1$, 179-235.

\bibitem{Paicu2}
M. Paicu, \textit{Equation anisotrope de Navier-Stokes dans des espaces critiques}, Revista Mathematica Iberoamericana  21  (2005), $n^o1$, 179-235.

\bibitem{pedlo} 
J. Pedlosky: {\it Geophysical Fluid Dynamics}, Springer Verlag, New-York, 1987.

\bibitem{Wu}
J. Wu, X. Xu, Z. Ye: \textit{The 2D Boussinesq equations with fractional horizontal dissipation and thermal diffusion}, 
Journal of Differential Equations 115, (2018), 187-217.


\end{thebibliography}
\end{document}